\documentclass[12pt]{amsart}

\usepackage[english]{babel}
\usepackage[pdftex]{graphicx}
\usepackage{amscd}
\usepackage{amsmath}
\usepackage{amsfonts}
\usepackage{amssymb}
\usepackage{amsthm}
\usepackage{enumerate}
\usepackage{color}
\usepackage{url}
\usepackage{hyperref}
\usepackage{verbatim} 

\theoremstyle{plain}
\newtheorem{theorem}{Theorem}[section]
\newtheorem{algorithm}[theorem]{Algorithm}
\newtheorem{lemma}[theorem]{Lemma}
\newtheorem{proposition}[theorem]{Proposition}
\theoremstyle{remark}
\newtheorem{remark}[theorem]{Remark}
\theoremstyle{definition}
\newtheorem*{acknowledgement}{Acknowledgement}

\newcommand{\thmref}[1]{Theorem~\ref{#1}}

\newcommand{\lemref}[1]{Lemma~\ref{#1}}

\numberwithin{equation}{section}

\newcommand{\Ind}[1]{\mathbf{1}_{\left\{#1\right\}}}
\newcommand{\PP}{\mathbb{P}}
\newcommand{\NN}{\mathbb{N}}
\newcommand{\EE}{\mathbb{E}}
\newcommand{\FF}{\mathbb{F}}
\newcommand{\GG}{\mathbb{G}}

\newcommand{\ZZ}{\mathbb{Z}}
\newcommand{\RR}{\mathbb{R}}
\newcommand{\cA}{\mathcal{A}}
\newcommand{\cB}{\mathcal{B}}
\newcommand{\cC}{\mathcal{C}}
\newcommand{\cD}{\mathcal{D}}
\newcommand{\cE}{\mathcal{E}}

\newcommand{\cG}{\mathcal{G}}
\newcommand{\cI}{\mathcal{I}}
\newcommand{\cL}{\mathcal{L}}
\newcommand{\cP}{\mathcal{P}}
\newcommand{\cT}{\mathcal{T}}
\newcommand{\cU}{\mathcal{U}}
\newcommand{\cV}{\mathcal{V}}
\newcommand{\cW}{\mathcal{W}}
\newcommand{\cX}{\mathcal{X}}

\newcommand{\degr}{\operatorname{deg}}
\newcommand{\dist}{\operatorname{dist}}
\newcommand{\bQ}{\mathbf{Q}}
\newcommand{\bP}{\mathbf{P}}

\newcommand{\capa}{\operatorname{cap}}
\newcommand{\Var}{\operatorname{Var}}
\newcommand{\floor}[1]{\lfloor #1 \rfloor}

\begin{document}

\title[Vacant set of random walk on giant component]{Phase transition for the vacant set left by random walk on the giant component of a random graph}
\author[]{Tobias Wassmer}
\address{\tiny University of Vienna\\Faculty of Mathematics\\Oskar-Morgenstern-Platz 1\\A-1090 Wien\\Austria}
\email{\tiny \href{mailto:tobias.wassmer@univie.ac.at}{tobias.wassmer@univie.ac.at}}

\begin{abstract}
We study the simple random walk on the giant component of a supercritical Erd\H{o}s-R\'enyi random graph on $n$ vertices, in particular the so-called vacant set at level $u$, the complement of the trajectory of the random walk run up to a time proportional to $u$ and $n$. We show that the component structure of the vacant set exhibits a phase transition at a critical parameter $u_{\star}$: For $u<u_{\star}$ the vacant set has with high probability a unique giant component of order $n$ and all other components small, of order at most $\log^{7}n$, whereas for $u>u_{\star}$ it has with high probability all components small. Moreover, we show that $u_{\star}$ coincides with the critical parameter of random interlacements on a Poisson-Galton-Watson tree, which was identified in \cite{Tas10}.
\end{abstract}

\date{\today}

\maketitle
\thispagestyle{empty}


\section{introduction}

Recently, several authors have been studying percolative properties of the vacant set left by random walk on finite graphs and the connections of this problem to the random interlacements model introduced in \cite{Sz10}. The topic was initiated  with the study of random walk on the $d$-dimensional discrete torus in \cite{BS08}, which was further investigated in \cite{TW11}. \cite{CTW11}, \cite{CT11} and \cite{CF11} studied random walk on the random regular graph, and \cite{CF11} also studied random walk on the Erd\H{o}s-R\'enyi random graph  above the connectivity threshold.

In this work we consider the supercritical Erd\H{o}s-R\'enyi random graph below the connectivity threshold. We prove a phase transition in the component structure of the vacant set left by random walk on the giant component of this graph, and we identify the critical point of this phase transition with the critical parameter of random interlacements on a Poisson-Galton-Watson tree.

We start by introducing some notation to precisely state the result. Let $\PP_{n,p}$ be the law of an Erd\H{o}s-R\'enyi random graph, i.e.~a random graph $G$ such that every possible edge is present independently with probability $p=\frac{\rho}{n}$, defined on the space $\cG(n)$ of graphs with vertex set $\{1,2,...,n\}$ endowed with the $\sigma$-algebra $\GG_n$ of all subsets. It is well known that the component structure of $G$ varies with the parameter $\rho$ (see e.g.~\cite{ER61}, \cite{Bol01}, \cite{JLR00}, \cite{Dur10}). We will in this paper consider such a random graph for a fixed constant $\rho>1$. In this case, with probability tending to 1 as $n\to\infty$, the graph $G$ is supercritical: There exists a unique largest connected component $\cC_1(G)$ of size approximately $\xi n$, the so-called giant component. Here, $\xi$ is the unique solution in $(0,1)$ of $e^{-\rho\xi}=1-\xi$.

For a graph $G$ on $n$ vertices and its largest connected component $\cC_1=\cC_1(G)$ (determined by some arbitrary tie-breaking rule), let $P^{\cC_1}$ be the law of the simple discrete-time random walk $(X_k)_{k\geq0}$ on $\cC_1$ started from its stationary distribution, defined on the space $\{1,2,...,n\}^{\NN_0}$ of trajectories on $n$ vertices endowed with the cylinder-$\sigma$-algebra $\FF_n$. Let $\Omega_n=\cG(n)\times\{1,2,...,n\}^{\NN_0}$ endowed with the product $\sigma$-algebra $\GG_n \times \FF_n$, and define the ``annealed'' measure by
\begin{equation} \bP_n(A\times B)=\sum_{G\in A} \PP_{n,p}(G) P^{\cC_1(G)}(B) \qquad \text{for }A\in\GG_n,~B\in\FF_n. \label{def:annealed} \end{equation}
On the product space $\Omega_n$ we define the vacant set of the random walk at level $u$ as
\begin{equation} \cV^{u}=\cC_1 \setminus \{X_k:~0\leq k \leq u\rho(2-\xi)\xi n\}. \label{def:vset} \end{equation}
We refer to Remark~\ref{rem} for an explanation of this somewhat unusual time scaling. Let $\cC_1(\cV^u)$ and $\cC_2(\cV^u)$ be the largest and second largest connected components of the subgraph induced by $\cV^u$.

\begin{theorem} \label{thm1}
The component structure of the subgraph induced by $\cV^{u}$ exhibits a phase transition at a critical value $u_{\star}$:
\begin{itemize}
\item For $u<u_{\star}$, there are positive constants $\zeta(u,\rho)\in(0,1)$, $C<\infty$, such that for every $\epsilon>0$,
\begin{gather} \lim_{n\to\infty} \bP_n\left[\left|\frac{|\cC_1(\cV^u)|}{n}-\zeta(u,\rho)\right|\leq \epsilon \right] = 1, \label{eq:thmeq1} \\
								\lim_{n\to\infty} \bP_n\left[\frac{|\cC_2(\cV^u)|}{\log^{7}n}\leq C \right] = 1. \label{eq:thmeq2} \end{gather}
\item For $u>u_{\star}$, there is a positive constant $C<\infty$, such that
\begin{equation} \lim_{n\to\infty} \bP_n\left[\frac{|\cC_1(\cV^u)|}{\log^{7}n}\leq C \right] = 1. \label{eq:thmeq3} \end{equation}
\end{itemize}
The critical parameter $u_{\star}$ is the same as the critical parameter of random interlacements on a Poisson($\rho$)-Galton-Watson tree conditioned on non-extinction, which is by \cite{Tas10} given as the solution of a certain equation.
\end{theorem}

We refer to Section~\ref{sec:ri} for a short summary of the used results on random interlacements and its critical parameter, and the derivation of the characterizing equation~\eqref{eq:ustarri} for $u_{\star}$. The constant $\zeta(u,\rho)$ is given as the solution of equation~\eqref{eq:zeta}.

\vspace{\baselineskip}
\thmref{thm1} confirms the following general principle: The vacant set of random walk on a sufficiently fast mixing graph exhibits a phase transition and the critical point is related to the critical value of random interlacements on the corresponding infinite volume limit.

This principle has been investigated recently in several other situations. Results that are more detailed than \thmref{thm1} are known to hold for random walk on a random $d$-regular graph on $n$ vertices run up to time $un$: \cite{CTW11} and with different methods \cite{CF11} proved the phase transition in the component structure of the vacant graph, \cite{CTW11} identified the critical parameter $u_{\star}$ with the critical value of random interlacements on the infinite $d$-regular tree, and \cite{CT11} showed that there is a critical window of width $n^{-\frac{1}{3}}$ around $u_{\star}$ in which the largest component is of order $n^{\frac{2}{3}}$. \cite{CF11} used their methods to also prove a phase transition for random walk on the Erd\H{o}s-R\'enyi random graph above the connectivity threshold ($\rho \gg \log n$). Weaker statements are known for random walk run up to time $uN^d$ on the discrete $d$-dimensional torus of sidelength $N$, see \cite{BS08} and \cite{TW11}. The statements in this case are proved for $u$ small or large enough respectively, but it is only conjectured that there is indeed a phase transition at a critical parameter $u_{\star}$ that coincides with the critical value of random interlacements on $\ZZ^d$ (cf.~Conjecture 2.6 in \cite{CT12}). We believe that in our case, as in \cite{CT11} for the random regular graph, it should be possible to prove the existence of a critical window around the critical point. We did not further investigate this.

\vspace{\baselineskip}
The main difficulties in proving \thmref{thm1} compared to previous results are that our graph, i.e.~the giant component of an Erd\H{o}s-R\'enyi random graph, is of random size and non-regular. The proof consists of three main steps. The key idea of the first step is the following ``spatial Markov property'' of random walk on a random graph. Instead of sampling a random graph and performing random walk on the fixed graph, one can consider sites unvisited by the random walk as not yet sampled sites of the random graph. Then the unvisited or vacant part of the graph has the law of some random graph, depending on the random graph model. In the case of a connected Erd\H{o}s-R\'enyi random graph the vacant part is again an Erd\H{o}s-R\'enyi random graph, this was used to prove the phase transition in \cite{CF11}. In the case of a random regular graph the vacant part is a random graph with a given degree sequence, a well-studied object (see e.g.~\cite{HM12}). This was used to prove the phase transition in \cite{CF11} and the critical behaviour in \cite{CT11}. 

The situation in our case is more involved, because we consider random walk only on the giant component of a not connected Erd\H{o}s-R\'enyi random graph. This random walk cannot satisfy such a spatial Markov property, since the graph must be fixed in advance for the giant component to be known. To be able to still use the idea, we introduce in Algorithm~\ref{algo} a process $\bar{X}=(\bar{X}_k)_{k\geq0}$ on an Erd\H{o}s-R\'enyi random graph that behaves like a random walk but jumps to another component after having covered a component. In \lemref{lem:defdec} we make precise the aforementioned spatial Markov property for this process $\bar{X}$, namely that the vacant graph left by $\bar{X}$ still has the law of an Erd\H{o}s-R\'enyi random graph, but with different parameters. The classical results on random graphs imply a phase transition for this vacant graph.

In a second step we translate this phase transition to the vacant graph left by the simple random walk $X=(X_k)_{k\geq0}$ on the giant component. To this end, we introduce in Proposition~\ref{prop:coupling} a coupling of $X$ and $\bar{X}$ where the two processes are with high probability identified in a certain time interval. This can be done because the process $\bar{X}$ will typically ``find'' the giant component after a short time and then stay on it long enough.

The third step, requiring most of the technical work, is the identification of the critical point of the phase transition. From \lemref{lem:defdec} it is clear that the crucial quantity deciding the critical point is the size of the vacant set left by $\bar{X}$. The coupling of $X$ and $\bar{X}$ has the property that the sizes of the vacant sets of $X$ and $\bar{X}$ are closely related (\lemref{lem:sizescomp}), which allows to reduce the problem to the investigation of the size of the vacant set left by $X$. The first part of this paper, Section~\ref{sec:size}, is devoted to this investigation. In Proposition~\ref{prop:sizeexpconc} we will on one hand compute the expectation of the size of the vacant set left by $X$, and on the other hand we will show that the size of the vacant set left by $X$ is concentrated around its expectation.

\vspace{\baselineskip}
We close the introduction with a remark on the connection to random interlacements and a heuristic explanation of the time scaling $u\rho(2-\xi)\xi n$ that appears in the definition \eqref{def:vset} of $\cV^u$. For readers unfamiliar with random interlacements and the notation, we refer to Section~\ref{sec:not}, in particular Section~\ref{sec:ri}.

\begin{remark} \label{rem}
In the giant component $\cC_1$ of an Erd\H{o}s-R\'enyi random graph the balls $B(x,r)$ around a vertex $x$ with radius $r$ of order $\log n$ typically look like balls around the root $\varnothing$ in a Poisson($\rho$)-Galton-Watson tree $\cT$ conditioned on non-extinction. One expects that random interlacements on $\cT$ give a good description of the trace of random walk on $\cC_1$ locally in such balls, where the intensity $u$ of random interlacements is proportional to the running time of the walk. To determine the proportionality factor, we compare the probability that a vertex $x\in\cC_1$ has not been visited by the random walk on $\cC_1$ up to time $t$ with the probability that the root $\varnothing\in\cT$ is in the vacant set of random interlacements on $\cT$ at level $u$.

Note first that the probability that the random walk on $\cC_1$ started at $x$ leaves a ball of large radius around $x$ before returning to $x$ is approximately the same as the probability that the random walk on $\cT$ started at the root never returns to the root,
\begin{equation} P^{\cC_1}_x[\tilde{H}_x > H_{B(x,r)^c}]\approx P_{\varnothing}^{\cT}[\tilde{H}_{\varnothing}=\infty]. \label{eq:escapeprob} \end{equation}
The main task of Section~\ref{sec:size} will be rigorous proof of the following approximation for the random walk on $\cC_1$,
\begin{equation} P^{\cC_1}[x \text{ is vacant at time }t] \approx e^{-t P^{\cC_1}_x[\tilde{H}_x > H_{B(x,r)^c}] \pi(x)}. \label{eq:vacantrw} \end{equation}
We will also show that the average degree of a vertex in $\cC_1$ is $\rho(2-\xi)$, and so the stationary distribution $\pi$ of the random walk on $\cC_1$ is $\pi(x) \approx \frac{\degr(x)}{\rho(2-\xi)\xi n}$. On the other hand, according to \cite{Tei09}, the law $Q^u$ of the vacant set of random interlacements on the infinite graph $\cT$ at level $u$ satisfies
\begin{equation} Q^u[\varnothing \text{ is vacant}] = e^{-u \capa_{\cT}(\varnothing)}, \label{eq:vacantri} \end{equation}
where the capacity is here $\capa_{\cT}(\varnothing)=\degr(\varnothing) P_{\varnothing}^{\cT}[\tilde{H}_{\varnothing}=\infty]$. As argued above, random interlacements describe the random walk locally, so the probabilites \eqref{eq:vacantrw} and \eqref{eq:vacantri} should be approximately equal for the time $t$ corresponding to random interlacements at level $u$. The approximation of $\pi(x)$ together with \eqref{eq:escapeprob} leads to $t=u\rho(2-\xi)\xi n$ if the parameter $u$ in both models should be the same.

Compared to the time scalings $uN^d$ and $un$ in the discussions of random walk on the torus (\cite{BS08}, \cite{TW11}) and random regular graphs (\cite{CTW11}, \cite{CT11}) respectively, where only the size of the graph (in our case the factor $\xi n$) appears in the time scaling, the additional factor $\rho(2-\xi)$ for the average degree might be surprising. It is however only a consequence of how one defines the uniform edge-weight on the underlying graph, which scales the capacity by a constant. For the aforementioned $2d$-regular graphs the weight chosen is $\frac{1}{2d}$. For non-regular graphs it is the canonical choice to define edge weights as 1, as is done in \cite{Tei09} and \cite{Tas10}, and we stick to this definition.
\end{remark}

\vspace{\baselineskip}
The paper is structured as follows. In Section~\ref{sec:not} we introduce some further notation and recall some facts on random graphs, random walks, and random interlacements. In Section~\ref{sec:size} we investigate the size of the vacant set left by the simple random walk $X$ on the giant component. In Section~\ref{sec:coupling} we introduce the process $\bar{X}$ and compare it to the random walk $X$. Finally, we gather all intermediate results to prove \thmref{thm1} in Section~\ref{sec:proof}.

\vspace{\baselineskip}
\begin{acknowledgement} The author would like to thank Ji\v{r}\'{i} \v{C}ern\'{y} for suggesting the problem and for helpful discussions, and the referee for carefully reading the manuscript and giving important comments that helped to improve the paper. \end{acknowledgement}


\vspace{\baselineskip}
\section{Notations and preliminaries} \label{sec:not}

We will denote by $c$, $c'$, $c''$ positive finite constants with values changing from place to place. $\epsilon$ will always denote a small positive constant with value changing from place to place. All these constants may depend on $u$ and $\rho$, but not on any other object. We will tacitly assume that values like $u\rho(2-\xi)\xi n$, $\log^5n$, $n^{\epsilon}$ etc.~are integers, omitting to take integer parts to ease the notation.

We use the standard $o$- and $O$-notation: Given a positive function $g(n)$, a function $f(n)$ is $o(g)$ if $\lim_{n\to\infty} f/g =0$, and it is $O(g)$ if $\limsup_{n\to\infty} |f|/g < \infty $. We extend this notation to random variables in the following way. For a random variable $A_n$ on a space $(\Omega_n,Q_n)$ we use the notation ``$A_n=f(n)+o(g)$ $Q_n$-asymptotically almost surely'' meaning ``$\forall~\epsilon>0$, \mbox{$Q_n[|A_n-f(n)|\leq\epsilon g(n)]\to1$} as $n\to\infty$'', and ``$A_n=O(g)$ $Q_n$-asymptotically almost surely'' meaning ``$\exists~C>0$ such that $Q_n[|A_n|\leq Cg(n)]\to1$ as $n\to\infty$''.

\subsection{(Random) graphs}
For a non-oriented graph we use the notation $G$ to denote the set of vertices in the graph as well as the graph itself, consisting of vertex-set and egde-set. For vertices $x,y\in G$, $x\sim y$ means that $x$ and $y$ are neighbours, i.e.~$\{x,y\}$ is an edge of $G$. We denote by $\degr(x)$ the number of neighbours of $x$ in $G$, and by $\Delta_G=\max_{x\in G} \degr(x)$ the maximum degree. By $\dist(x,y)$ we denote the usual graph distance, and for $r\in\NN$, $B(x,r)$ is the set of vertices $y$ with $\dist(x,y)\leq r$. For a subset $A\subset G$, denote its complement $A^c=G\setminus A$ and its (interiour) boundary $\partial A = \{x\in A:~\exists y\in A^c,~x\sim y\}$.

We denote by $\cC_i(G)$ the $i$-th largest connected component of a graph $G$. If there are equally large components, we order these arbitrarily. The subgraph induced by a vertex-set $V\subset G$ is defined as the graph with vertices $V$ and edges $\{x,y\}$ if and only if $x,y\in V$ and $x\sim y$ in $G$. Again we use the notation $\cC_i(G)$ for the set of vertices as well as for the induced subgraph. Usually (but not necessarily) $\cC_1=\cC_1(G)$ will be the unique giant component. A graph or graph component is called ``simple'' if it is connected and has at most one cycle, i.e.~the number of edges is at most equal to the number of vertices. 

Recall from the introduction that $\PP_{n,p}$ denotes the law of an Erd\H{o}s-R\'enyi random graph, i.e.~a random graph on $n$ vertices such that every edge is present independently with probability $p=\frac{\rho}{n}$. Let $\EE_{n,p}$ be the corresponding expectation. An event is said to hold ``asymptotically almost surely'' (a.a.s.) if it holds with probability tending to 1 as $n\to\infty$ (cf.~the above defined $o$- and $O$-notation). Throughout this work $\rho>1$ is a fixed constant. It is well known that the following properties then hold $\PP_{n,p}$-a.a.s.
\begin{align}
&\parbox{15.1cm}{The graph $G$ has a unique giant component $\cC_1$ of size $|\cC_1|$ satisfying $\left| |\cC_1| - \xi n \right| \leq n^{3/4}$, where $\xi$ is the unique solution in $(0,1)$ of $e^{-\rho\xi}=1-\xi$. All other components are simple and of size smaller than $C \log n$, for some fixed constant $C$.} \label{prop:giant} \\
&\parbox{15.1cm}{The spectral gap $\lambda_{\cC_1}$ of the random walk on the giant component (cf.~\eqref{def:gap}) satisfies $\lambda_{\cC_1}\geq \frac{c}{\log^2n}$ for some fixed constant $c$.} \label{prop:gap} \\
&\parbox{15.1cm}{The maximum degree $\Delta_G$ satisfies $\Delta_G \leq \log n$.} \label{prop:maxdeg}
\end{align}
\eqref{prop:giant} and \eqref{prop:maxdeg} are classical results (see e.g.~\cite{ER61}, \cite{Bol01}, \cite{JLR00} or \cite{Dur10}), and \eqref{prop:gap} follows from \cite[Theorem 12.4 ]{LPW09} with the $O(\log^2n)$ bound on the mixing time of the random walk on the giant component proved in \cite{BKW06}. We use the terminology ``typical graphs'' for graphs $G$ on $n$ vertices satisfying \eqref{prop:giant}, \eqref{prop:gap} and \eqref{prop:maxdeg}. We will usually prove our statements for typical graphs only, since we are interested in a.a.s.-behaviour.

For a quantitative version of the first statement in \eqref{prop:giant} see \cite[Theorem 4.8]{vdH08}, which states that
\begin{equation} \PP_{n,p} \left[\left| |\cC_1| - \xi n \right| > n^{3/4}\right] \leq cn^{-c'}. \label{prop:giantquant} \end{equation}
The choice of the constant 3/4 is arbitrary.

We will also need a quantitative version of \eqref{prop:maxdeg}, we therefore briefly present a proof. Fix a vertex $x\in G$ and denote all other vertices by $y_i$, $i=1,...,n-1$. Let $\cE_i=\Ind{\{x,y_i\}\text{ is an edge}}$. Then the $\cE_i$ are i.i.d.~Bernoulli($p$) random variables, $\degr(x)=\sum_{i=1}^{n-1} \cE_i$, and for any fixed $\alpha>0$ by the exponential Chebyshev inequaliy,
\begin{align*} \PP_{n,p}[\degr(x)>\log n] \leq n^{-\alpha}\EE_{n,p}\left[e^{\alpha\sum \cE_i}\right] = n^{-\alpha}\left(1+\frac{\rho}{n}(e^{\alpha}-1)\right)^{n-1} \leq cn^{-\alpha}, \end{align*}
where the constant $c$ depends on $\alpha$. We choose $\alpha=4$, this will be suitable for our purposes. Then a union bound implies
\begin{equation} \PP_{n,p}[\Delta_G>\log n] \leq n \PP_{n,p}[\degr(x)>\log n] \leq cn^{1-\alpha}= cn^{-3}. \label{eq:pmaxdeg} \end{equation}

\subsection{Random walks}
Let $P^{\cC_1}$ be the law and $E^{\cC_1}$ the corresponding expectation of the simple discrete-time random walk $X=(X_k)_{k\geq0}$ on the component $\cC_1$ started stationary, i.e.~the law of the Markov chain with state space $\cC_1$, transition probabilities $p_{xy}=\frac{1}{\deg(x)}\Ind{x\sim y}$ and $X_0 \sim \pi$, where $\pi$ is the stationary distribution, $\pi(x)=\frac{\degr(x)}{\sum_{y\in \cC_1} \degr(y)}$. \eqref{prop:giant} and the a.a.s.~upper bound \eqref{prop:maxdeg} on the maximum degree $\Delta_G$ imply the following bounds on $\pi$. $\PP_{n,p}$-a.a.s.
\begin{align} \pi(x) &= \frac{\degr(x)}{\sum_{v\in \cC_1} \degr(v)} \leq \frac{c\log n}{n} \label{eq:piu}, \\
							\pi(x) &= \frac{\degr(x)}{\sum_{v\in \cC_1} \degr(v)} \geq \frac{c}{n\log n}. \label{eq:pil} \end{align}
							
For real numbers $0\leq s\leq r$ denote by $X_{[s,r]}=\{X_k:~s\leq k \leq r\}$ the set of vertices visited by $X$ between times $s$ and $r$. We let the random walk $X$ run up to time $t$ and denote by $\cV(t)=\cC_1\setminus X_{[0,t]}$ the vacant set left by the random walk at time $t$, and again we use the notation $\cV(t)$ to also denote the subgraph of $\cC_1$ induced by these vertices. As defined in \eqref{def:vset}, we will use the short notation $\cV^u$ for $\cV(u\rho(2-\xi)\xi n)$.

We will, where it is clear in the context, drop the superscript from $P^{\cC_1}$ and $E^{\cC_1}$. The notation $P_x$ is then used to denote the law of the random walk on $\cC_1$ started at vertex $x$, $E_x$ is the corresponding expectation. For a set $A\subset \cC_1$ we denote by
\begin{equation*} H_A=\inf\{t\geq0:~X_t\in A\}, \qquad \tilde{H}_A=\inf\{t\geq1:~X_t\in A\}\end{equation*}
the entrance time and hitting time respectively of $A$, and we write $H_x$ and $\tilde{H}_x$ if $A=\{x\}$. From \cite[Lemma 2]{AB92} or \cite[Chapter 3, Proposition 21]{AFb} together with \eqref{eq:piu} we get the following bound on $E[H_x]$. $\PP_{n,p}$-a.a.s.~for all $x\in \cC_1$,
\begin{equation} E[H_x] \geq \frac{(1-\pi(x))^2}{\pi(x)} \geq \frac{cn}{\log n}. \label{eq:epihl} \end{equation}

For all real valued functions $f$ and $g$ on $\cC_1$ define the Dirichlet form
\begin{equation} \cD(f,g)=\frac{1}{2} \sum_{x,y\in \cC_1} (f(x)-f(y))(g(x)-g(y)) \pi(x) p_{xy}. \label{def:D} \end{equation}
A function $f$ on $\cC_1$ is harmonic on $A\subset\cC_1$ if $\sum_y p_{xy} f(y) = f(x)$ for $x\in A$. For $x\in \cC_1$ and $r\in\NN$ define the equilibrium potential $g^{\star}:\cC_1\to\RR$ as the unique function harmonic on $B(x,r)\setminus \{x\}$, 1~on~$\{x\}$ and 0 on $B(x,r)^c$. The dependence of $g^{\star}$ on $x$ and $r$ is kept implicit. Then it is well known that
\begin{align} g^{\star}(y) &= P_y\left[H_x < H_{B(x,r)^c}\right], \label{eq:gstar} \\
							\cD(g^{\star},g^{\star}) &= P_x\left[\tilde{H}_x > H_{B(x,r)^c}\right]\pi(x). \label{eq:Dgstar} \end{align}

The spectral gap of the random walk on $\cC_1$ is given by
\begin{equation} \lambda_{\cC_1} = \min\{\cD(f,f):~ \pi(f^2)=1,~\pi(f)=0\}. \label{def:gap} \end{equation}
The relevance of the bound \eqref{prop:gap} on $\lambda_{\cC_1}$ is in the speed of mixing of the random walk on $\cC_1$. From \cite[Theorem 12.3 and Lemma 6.13]{LPW09} it follows that for all $t\in\NN$
\begin{equation} \max_{x,y \in \cC_1}\left|P_x[X_t=y]-\pi(y)\right| \leq \frac{1}{\min_{z\in \cC_1} \pi(z)} e^{-\lambda_{\cC_1}t}. \label{eq:mix} \end{equation}

\subsection{Random interlacements} \label{sec:ri}
Random interlacements were introduced in \cite{Sz10} on $\ZZ^d$ as a model to describe the local structure of the trace of a random walk on a large discrete torus, and in \cite{Tei09} the model was generalized to arbitrary transient graphs. It is a special dependent site-percolation model where the occupied vertices on a graph are constructed as the trace left by a Poisson point process on the space of doubly infinite trajectories modulo time shift. The density of this Poisson point process is determined by a parameter $u>0$. The critical value $u_{\star}$ is the infimum over the $u$ for which almost surely all connected components of non-occupied vertices are finite.

In \cite{Tas10} it is shown that for random Galton-Watson trees the critical value $u_{\star}$ is almost surely constant with respect to the tree measure and is implicitly given as the solution of a certain equation. Except for the identification of the critical parameter of \thmref{thm1} with this $u_{\star}$ as the solution of the same equation, we will not use any results on random interlacements. We refer to the lecture notes \cite{CT12} for an introduction to random interlacements and many more references.

We quote the result from \cite{Tas10} to derive the characterizing equation for $u_{\star}$ in the case of a Poisson-Galton-Watson tree. This requires some more notation. Denote by $\PP_{\cT}$ the law of the supercritical Poisson($\rho$)-Galton-Watson rooted tree conditioned on non-extinction and by $\EE_{\cT}$ the corresponding conditional expectation. Let $f(s)=e^{\rho(s-1)}$ be the probability generating function of the Poisson($\rho$) distribution, and denote by $q$ the extinction probability of a (unconditioned) Poisson($\rho$)-Galton-Watson tree. It is well known that $q$ is the unique solution in $(0,1)$ of the equation $f(s)=s$, and hence $q=1-\xi$, where $\xi$ is as in \eqref{prop:giant}. Let 
\begin{equation} \tilde{f}(s) = \frac{f((1-q)s+q) -q}{1-q}. \label{eq:tildef} \end{equation}
This is in fact the probability generating function of the offspring in the subtree of vertices with infinite line of descent (see e.g.~\cite[Proposition 5.26]{LP}).

Consider the simple discrete-time random walk $(X_k)_{k\geq0}$ on the rooted tree $\cT$ started at the root $\varnothing$, whose law we denote by $P^{\cT}_{\varnothing}$, and let $\tilde{H}_{\varnothing}=\inf\{t\geq1:~X_t=\varnothing\}$ be the hitting time of the root. Define the capacity of the root by $\capa_{\cT}(\varnothing)=\degr(\varnothing) P_{\varnothing}^{\cT}[\tilde{H}_{\varnothing}=\infty]$. 

By \cite[Theorem 1]{Tas10}, the critical parameter $u_{\star}$ of random interlacements on the Galton-Watson tree conditioned on non-extinction is $\PP_{\cT}$-a.s.~constant and given as the unique solution in $(0,\infty)$ of the equation
\begin{equation*} \left( \tilde{f}^{-1}\right)' \left( \EE_{\cT}\left[e^{-u \capa_{\cT}(\varnothing)}\right] \right) =1. \end{equation*}
In particular for the Poisson($\rho$)-Galton-Watson tree,
\begin{equation*} \left( \tilde{f}^{-1}\right)'(t) = \frac{1}{\rho\xi t + \rho(1-\xi)}, \end{equation*}
and $u_{\star}$ is the solution of 
\begin{equation} \rho \xi \EE_{\cT}\left[e^{-u \capa_{\cT}(\varnothing)}\right]+\rho(1-\xi)=1. \label{eq:ustarri} \end{equation}


\vspace{\baselineskip}
\section{Size of the vacant set} \label{sec:size}

In this section we investigate the size of the vacant set $\cV^{u}$ left by the random walk $X$ on the giant component $\cC_1$. As already mentioned we omit the superscripts from $P^{\cC_1}$ and $E^{\cC_1}$. Recall the definition \eqref{def:annealed} of the annealed measure $\bP_n$.

\begin{proposition} ~ \label{prop:sizeexpconc}
\begin{enumerate}
\item $E[|\cV^{u}|]$ can asymptotically be approximated in terms of a Poisson($\rho$)-Galton-Watson tree conditioned on non-extinction:
\begin{equation*} E[|\cV^{u}|] = \xi n \EE_{\cT}\left[e^{-u \capa_{\cT}(\varnothing)}\right] +o(n) \qquad \PP_{n,p} \text{-a.a.s.} \end{equation*}
\item The random variable $|\cV^{u}|$ is concentrated around its mean:
\begin{equation*} |\cV^{u}| = E[|\cV^{u}|] + o(n) \qquad \bP_n \text{-a.a.s.} \end{equation*}
\end{enumerate}
\end{proposition}


\vspace{\baselineskip}
\subsection{Expectation of the size of the vacant set}
The proof of part (1) of Proposition~\ref{prop:sizeexpconc} is split up into several steps. We first quote and extend \cite[Proposition 11.2]{JLT12}. It formalizes the well known fact that an Erd\H{o}s-R\'enyi random graph locally looks like a Galton-Watson tree. Here, by locally we mean balls of radius of order $\log n$. More precisely, fix some $\gamma>0$ such that $6\gamma\log\rho<1$, and set
\begin{equation} r= \gamma \log n. \label{def:r} \end{equation}
For a graph $G$, a vertex $x\in G$ and a tree $\cT$ with root $\varnothing$, define the event
\begin{equation} \cI_x(G,\cT)=\left\{\parbox{9cm}{ $B(x,r+1)\subset G$ is isomorphic to $B(\varnothing,r+1)\subset \cT$, with the isomorphism sending $x$ to $\varnothing$} \right\}. \label{def:isom} \end{equation}
Denote by $\PP^0_{\cT}$ the law of the unconditioned Poisson($\rho$)-Galton-Watson tree $\cT$, and by $\{|\cT|<\infty\}$, $\{|\cT|=\infty\}$ the events of extinction and non-extinction respectively of the tree $\cT$. 

\pagebreak
\begin{proposition}~\label{prop:localtree}
\begin{enumerate}
\item Given an arbitrary fixed vertex $x\in\{1,2,...,n\}$, there is a coupling $Q_x$ of $G$ under $\PP_{n,p}$ and a tree $\cT$ under $\PP^0_{\cT}$, such that for $n$ large enough
\begin{equation} Q_x\left[\cI_x(G,\cT)\right] \geq 1-cn^{3\gamma\log\rho-1}. \label{eq:localtree} \end{equation}
For $n$ large enough, this coupling satisfies 
\begin{align} Q_x[x\in\cC_1,~|\cT|<\infty] &\leq cn^{-c'}, \label{eq:fail1} \\ Q_x[x\notin\cC_1,~|\cT|=\infty] &\leq cn^{-c'}. \label{eq:fail2} \end{align}
\item For an arbitrary point $x\in G$, with $r$ as in \eqref{def:r},
\begin{equation} \PP_{n,p}\left[|B(x,r)|\geq n^{3\gamma\log\rho}\right] \leq cn^{3\gamma\log\rho-1}. \label{eq:sizeball} \end{equation}
\item Given two arbitrary fixed vertices $x\neq y$, there is a coupling $Q_{x,y}$ of $G$ under $\PP_{n,p}$ and two trees $\cT_x$ and $\cT_y$, each having law $\PP^0_{\cT}$, such that $\cT_x$ and $\cT_y$ are independent and for $n$ large enough
\begin{equation} Q_{x,y}\left[\cI_x(G,\cT_x) \text{ and }\cI_y(G,\cT_y)\right] \geq 1-cn^{6\gamma\log\rho-1}, \label{eq:doublecoupling} \end{equation}
and statements \eqref{eq:fail1} and \eqref{eq:fail2} hold under $Q_{x,y}$ for $x$, $\cT_x$ and $y$, $\cT_y$ respectively.
\end{enumerate}
\end{proposition}

\begin{proof}
\eqref{eq:localtree} is, up to the enlargement of the radius by 1, the statement of \cite[Proposition 11.2]{JLT12}, and \eqref{eq:sizeball} is \cite[Corollary 11.3]{JLT12}. Note that, in contrary to the actual statement, \cite[Proposition 11.2]{JLT12} is proved for an a priori fixed vertex and not a randomly chosen one.

For part $(1)$ it remains to show the properties \eqref{eq:fail1} and \eqref{eq:fail2}. For simplicity write $B_x=B(x,r)\subset G$ and $B_{\varnothing}=B(\varnothing,r)\subset\cT$. Denote by $\{z\leftrightarrow B_z^c\}$ the event that $z$ is connected to the complement of $B_z$, or equivalently that $\partial B_z$ is non-empty, and by $\{z\not\leftrightarrow B_z^c\}$ its complement. To prove \eqref{eq:fail1}, we first claim that
\begin{equation} \PP_{n,p}[x\in\cC_1,~x\not\leftrightarrow B_x^c ] \leq cn^{-c'}. \label{eq:xc1notxtodb} \end{equation}
To see this, note that if $x\in\cC_1$ and $x\not\leftrightarrow B_x^c$, then $B_x=\cC_1$. But by \eqref{eq:sizeball}, $B_x$ is unlikely to be large: For every small $\epsilon>0$, $\PP_{n,p}[|B_x|\geq n^{1-\epsilon}]\leq cn^{-c'}$. However, if $B_x$ is smaller than $n^{1-\epsilon}$ and $B_x=\cC_1$, then $\cC_1$ is smaller than $n^{1-\epsilon}$, but this happens with probability smaller than $cn^{-c'}$ by \eqref{prop:giantquant}, and \eqref{eq:xc1notxtodb} follows.

Note that if the coupling succeeds, i.e.~the balls of radius $r+1$ are isomorphic, then $\{x\leftrightarrow B_x^c\}=\{\varnothing\leftrightarrow B_{\varnothing}^c\}$. This happens with probability $\geq 1- cn^{-c'}$ by \eqref{eq:localtree}, so together with \eqref{eq:xc1notxtodb},
\begin{align*} Q_x[x\in\cC_1,~|\cT|<\infty] &\leq Q_x[x\leftrightarrow B_x^c,~|\cT|<\infty]+cn^{-c'} \\ &\leq Q_x[\varnothing\leftrightarrow B_{\varnothing}^c,~|\cT|<\infty]+ cn^{-c'} = \PP^0_{\cT}[\varnothing\leftrightarrow B_{\varnothing}^c,~|\cT|<\infty]+ cn^{-c'}. \end{align*}
The tree $\cT$ conditioned on extinction has the law of a subcritical Galton-Watson tree with mean offspring number $m<1$ (see e.g.~\cite[Proposition 5.26]{LP}). If $q$ is the extinction probability and $Z_k$ denotes the size of the $k$-th generation of the tree, we can use the Markov inequality to get
\begin{align*} \PP^0_{\cT}[\varnothing\leftrightarrow B_{\varnothing}^c,~|\cT|<\infty] &= \PP^0_{\cT}\left[Z_{r} \geq 1~\big|~|\cT|<\infty\right]q \\ &\leq \EE^0_{\cT}\left[Z_{r}~\big|~|\cT|<\infty\right]q = qm^{\gamma \log n} = cn^{-c'}, \end{align*}
which proves \eqref{eq:fail1}.

For \eqref{eq:fail2}, let $C_x$ be the component of $G$ containing $x$. Let $M>0$ be such that $M\gamma>(\rho-1-\log \rho)^{-1}$. Then, by e.g.~\cite[Theorem 2.6.4]{Dur10}, $\PP_{n,p}[x\notin\cC_1,~ |C_x|>M\gamma\log n]\leq cn^{-c'}$. Using this on the first line and \eqref{eq:localtree} on the second, it follows that
\begin{equation*} \begin{split} Q_x[x\notin\cC_1,~|\cT|=\infty] &\leq Q_x[|C_x|\leq M\gamma\log n,~|\cT|=\infty]+cn^{-c'} \\ &\leq Q_x[|B_{\varnothing}|\leq M\gamma\log n,~|\cT|=\infty]+cn^{-c'}. \end{split} \end{equation*}

To bound this latter probability that the ball of radius $r=\gamma\log n$ in a surviving Poisson($\rho$)-Galton-Watson tree is smaller than $Mr$, let again $Z_r$ be the size of the $r$-th generation and denote by $Z_r^{\star}$ the number of particles in the $r$-th generation with infinite line of descent. Then
\begin{align*} Q_x[|B_{\varnothing}|\leq Mr,~|\cT|=\infty] &\leq \PP_{\cT}^0 [Z_r\leq Mr \mid |\cT|=\infty]\PP_{\cT}^0 [|\cT|=\infty] \\ &\leq \PP_{\cT}^0 [Z_r^{\star}\leq Mr \mid |\cT|=\infty]\xi. \end{align*}
By e.g.~\cite[Proposition 5.26]{LP} or \cite[Theorem I.12.1]{AN72}
\begin{equation*} \PP_{\cT}^0 [Z_r^{\star}\leq Mr \mid |\cT|=\infty]=\tilde{\PP}_{\cT}[\tilde{Z}_r\leq Mr], \end{equation*}
where $\tilde{Z_r}$ under $\tilde{\PP}_{\cT}$ is the $r$-th generation size of a Galton-Watson tree with offspring distribution defined by the probability generating function $\tilde{f}$ as in \eqref{eq:tildef}, a tree with extinction probability $\tilde{q}=0$. Let $\kappa=\tilde{f}'(0)=f'(q)$. Since $f$, the probability generating function of Poisson($\rho$), is strictly convex and increasing, and by definition of $q=1-\xi$, we have $0<\kappa<1$. Let $\tilde{f}_r$ be the $r$-th iterate of $\tilde{f}$, which is in fact the probability generating function of $\tilde{Z}_r$. From \cite[Corollary I.11.1]{AN72} we know that
\begin{equation*} \lim_{r\to\infty} \kappa^{-r} \tilde{f}_r(s) = Q(s)\in(0,\infty) \quad \text{exists for } 0\leq s <1. \end{equation*}
It follows that
\begin{equation*} \tilde{f}_r(s) \leq (Q(s) + \epsilon) \kappa^r \end{equation*}
for $r\geq r_0(s,\epsilon)$. Using this, for any $\lambda>0$ we obtain for $r\geq r_0(e^{-\lambda},\epsilon)$
\begin{align*} \tilde{\PP}_{\cT}[\tilde{Z}_r\leq Mr] &\leq \tilde{\PP}_{\cT}[e^{-\lambda\tilde{Z}_r}\geq e^{-\lambda Mr}] \leq e^{\lambda Mr} \tilde{f}_r(e^{-\lambda}) \\ &\leq (Q(s) + \epsilon) e^{\lambda Mr + r\log\kappa}. \end{align*}
By choosing $\lambda < -\frac{\log\kappa}{M}$ we can make this smaller than $ce^{-c'r}$, and \eqref{eq:fail2} follows since $r=\gamma\log n$. This finishes the proof of part $(1)$ of the proposition. 

We now prove part (3). Define the coupling $Q_{x,y}$ as follows. By using part (1) of the proposition, we can find a coupling of two independent graphs $G_x$ and $G_y$, both with vertex set $x,y,3,...,n$, and two independent Poisson($\rho$)-Galton-Watson trees $\cT_x$ and $\cT_y$, such that with probability larger than $1-2cn^{3\gamma\log\rho-1}$ both $\cI_x(G_x,\cT_x)$ and $\cI_y(G_y,\cT_y)$ hold.

We then construct a graph $G$ with the same vertex set $x,y,3,...,n$ in the following way. We first explore the ball $B(x,r+1)\subset G$ by determining the state of all possible edges with at least one adjacent vertex in $B(x,r)\subset G_x$ according to their state in $G_x$, i.e.~setting them present or absent. In a second step we determine the ball $B(y,r+1)\subset G$ in the same way by $G_y$, only that we do not change the state of already determined edges. The remaining edges in $G$ are set present independently with probability $p$ and absent otherwise.

By construction this graph $G$ has law $\PP_{n,p}$. If both $\cI_x(G_x,\cT_x)$ and $\cI_y(G_y,\cT_y)$ hold and there is no collision in the second step, i.e.~we never want to set an edge present that is already set absent or vice versa, then both $\cI_x(G,\cT_x)$ and $\cI_y(G,\cT_y)$ hold, and the coupling succeeds. It thus remains to bound the probability of such a collision.

Note that if there is a collision, then the sets of vertices $B(x,r+1)$ and $B(y,r+1)$ must have non-empty intersection: If $B(x,r+1)\cap B(y,r+1)=\emptyset$, the only edges possibly causing a collision are edges $\{u,v\}$ with $u\in B(x,r)$ and $v\in B(y,r)$, but these edges must be set absent by both $G_x$ and $G_y$, or else $u\in B(y,r+1)$ or $v\in B(x,r+1)$.

The sets $B(x,r+1)$ and $B(y,r+1)$ are smaller than $n^{3\gamma\log\rho}$ with probability larger than $1-cn^{3\gamma\log\rho-1}$ by \eqref{eq:sizeball}, and they are by construction random subsets of $\{x,y,3,...,n\}$. But the probability that two random subsets of $\{x,y,3,...,n\}$ of size $k$ intersect is smaller than $\frac{k^2}{n}$, so the probability of a collision is smaller than
\begin{equation*} Q_{x,y}[B(x,r+1)\cap B(y,r+1)\neq \emptyset] \leq 2cn^{3\gamma\log\rho-1} + \frac{1}{n} n^{6\gamma\log\rho} \leq cn^{6\gamma\log\rho-1}.\end{equation*}
This proves \eqref{eq:doublecoupling}. By construction it is clear that statements \eqref{eq:fail1} and \eqref{eq:fail2} hold analogously under $Q_{x,y}$.
\end{proof}

We will denote by $\EE_{Q_x}$ and $\EE_{Q_{x,y}}$ the expectations corresponding to the couplings $Q_x$ and $Q_{x,y}$. For easier use later we now define some events and estimate their probabilities. Let $\cB_x$ on the space of the coupling $Q_x$ be the event
\begin{equation} \cB_x = \cI_x(G,\cT) \cap \Big( \{x\in\cC_1,~|\cT|=\infty\} \cup \{x\notin\cC_1,~|\cT|<\infty\} \Big), \label{def:bx} \end{equation}
This event can canonically also be defined on the space of the coupling $Q_{x,y}$ when replacing $\cT$ by $\cT_x$. Then define on the space of $Q_{x,y}$ the event
\begin{equation} \cB_{x,y} = \cB_x \cap \cB_y. \label {def:bxy} \end{equation}
From Proposition~\ref{prop:localtree} it is immediate that
\begin{align} Q_x[\cB_x]\geq 1-cn^{-c'}, \label{eq:pbx} \\ Q_{x,y}[\cB_{x,y}]\geq 1-cn^{-c'}. \label{eq:pbxy} \end{align}
On the space of the coupling $Q_x$, and similarly on the space of $Q_{x,y}$, we further define the event
\begin{equation} \{x \text{ good}\} = \{x\in\cC_1\}\cap\{|\cT|=\infty\}\cap\cI_x(G,\cT) = \cB_x\cap\{x\in\cC_1\}. \label{def:xgood} \end{equation}
Since $\PP_{\cT}^0[|\cT|=\infty]=\xi$ and $\Ind{x \text{ good}}=\Ind{|\cT|=\infty} - \Ind{|\cT|=\infty,~x\notin\cC_1}-\Ind{|\cT|=\infty,~x\in\cC_1,~\cI_x(G,\cT)^c}$, it follows with \eqref{eq:localtree} and \eqref{eq:fail2} that
\begin{equation*} Q_x[x\text{ good}]=\xi +o(1) \text{ as } n\to\infty.  \end{equation*}

Note that the probability of $x$ being good is bounded away from zero, so every graph property holding $\PP_{n,p}$-a.a.s., as well as every property of a ball of radius $r$ in a Galton-Watson tree holding $\PP_{\cT}^0$-a.a.s.~as $r\to\infty$ will also hold $Q_x[~\cdot~|~x$~good$]$-a.a.s.

As a first application of Proposition~\ref{prop:localtree} we prove a law of large numbers for the sum of degrees of vertices in the giant component, which leads to an approximation of the stationary measure $\pi$. This result may be well known, we did however not find it in the literature. The technique of the proof will be used again later.

\begin{lemma} \label{lem:sumdeg}
\begin{equation*} \sum_{x\in \cC_1} \degr(x) = \sum_{x\in G} \Ind{x\in \cC_1} \degr(x) = \rho(2-\xi)\xi n + o(n) \qquad \PP_{n,p} \text{-a.a.s.} \end{equation*}
\end{lemma}

\begin{proof}
Every vertex in the random graph $G$ has Binomial($n-1,\frac{\rho}{n}$) neighbours, but on $\cC_1$ their degree is above average and there is some dependency. For $x\in G$ denote 
\begin{align*} Z_x & =\Ind{x\in \cC_1} \degr(x), \\ \tilde{Z}_x &= \Ind{|\cT|=\infty} \degr(\varnothing), \end{align*}
where the tree $\cT$ is defined by the coupling $Q_x$ from Proposition~\ref{prop:localtree}, and $\varnothing$ is the root of $\cT$. We will approximate $\EE_{n,p}[Z_x]=\EE_{Q_x}[Z_x]$ by $\EE_{Q_x}[\tilde{Z}_x]$ and show that the sum of the $Z_x$ is concentrated around its expectation using the second moment method.

Let us first compute the expectation of $\tilde{Z}_x$. Recall that $\PP_{\cT}$ denotes the law of the Poisson($\rho$)-Galton-Watson tree conditioned on non-extinction, and $\EE_{\cT}$ the corresponding conditional expectation. Then
\begin{equation} \EE_{Q_x}\left[\tilde{Z}_x\right] = \EE_{\cT}^0\left[\degr(\varnothing)~\big|~|\cT|=\infty\right]\PP^0_{\cT}\left[|\cT|=\infty\right] = \EE_{\cT}\left[\degr(\varnothing)\right]\xi. \label{eq:exptildez}\end{equation}
Using the same technique as in the proof of \cite[Proposition 5.26]{LP}, it is straightforward to see that the expected offspring in a Galton-Watson tree conditioned on non-exctinction is 
\begin{equation*} \EE_{\cT}[\degr(\varnothing)] = \frac{1}{1-q}(f'(1) - qf'(q)), \end{equation*}
where $f$ is the probability generating function of the offspring distribution. Here, the offspring is Poisson($\rho$), so $q=1-\xi$, $f'(1)=\rho$ and $f'(q)=\rho(1-\xi)$, which leads to
\begin{equation} \EE_{\cT}\left[\degr(\varnothing)\right] = \frac{1}{\xi}(\rho-\rho(1-\xi)^2) = \rho(2-\xi). \label{eq:expdegnonext} \end{equation}

We now approximate $\EE_{Q_x}[Z_x]$ by $\EE_{Q_x}[\tilde{Z}_x]$. Because $\tilde{Z}_x$ is unbounded, we will truncate it by $\log n$. By definition $\tilde{Z}_x$ is stochastically dominated by a Poisson($\rho$)-random variable $\Lambda$, in particular it has finite mean, and therefore $\EE_{Q_x}[\tilde{Z}_x \Ind{\tilde{Z}_x <\log n}] \nearrow \EE_{Q_x}[\tilde{Z}_x]$ as $n\to\infty$. Using $E[e^{t\Lambda}] = e^{\rho(e^t-1)}$ we have $P[\Lambda \geq \log n] = P[e^{t\Lambda}\geq n^t] \leq  e^{\rho(e^t-1)}n^{-t}=cn^{-c'}$. It follows that
\begin{align*}\EE_{Q_x}[\tilde{Z}_x\wedge\log n] &= \EE_{Q_x}[\tilde{Z}_x \Ind{\tilde{Z}_x <\log n}] + \log n Q_x[\tilde{Z}_x \geq \log n] \\ &= \EE_{Q_x}[\tilde{Z}_x] + o(1)\text{ as }n\to\infty.\end{align*}

Recall from \eqref{def:bx} the definition of the event $\cB_x$, on which $Z_x=\tilde{Z}_x$, and $Z_x=\tilde{Z}_x\wedge\log n$ if $\Delta_G \leq\log n$. With \eqref{eq:pbx} and \eqref{eq:pmaxdeg} we can bound
\begin{align} \left| \EE_{Q_x}[Z_x] - \EE_{Q_x}[\tilde{Z}_x\wedge\log n] \right| &\leq n Q_x\left[\Delta_G>\log n\right] + \log n Q_x\left[\cB_x^c\right] \leq cn^{-c'}. \label{eq:zfirsttilde} \end{align}
With \eqref{eq:exptildez} and \eqref{eq:expdegnonext} it follows that
\begin{equation*} \EE_{n,p}\left[\sum_{x\in G}Z_x\right] = n \EE_{Q_x}\left[Z_x\right] = \rho(2-\xi)\xi n + o(n) \text{ as } n\to\infty. \end{equation*}

It remains to show that the sum of the $Z_x$ is concentrated. Take $x\neq y$ arbitrary vertices in $G$ and consider the coupling $Q_{x,y}$ from Proposition~\ref{prop:localtree}. Recall from \eqref{def:bxy} the definition of the event $\cB_{x,y}$. On $\cB_{x,y}$ we have $Z_x=\tilde{Z}_x$ and $Z_y=\tilde{Z}_y$, so with \eqref{eq:pbxy} and \eqref{eq:pmaxdeg} we get
\begin{equation} \begin{split} &\left| \EE_{Q_{x,y}}[Z_xZ_y] - \EE_{Q_{x,y}}\left[(\tilde{Z}_x\wedge\log n)(\tilde{Z}_y\wedge\log n)\right] \right| \\ &\qquad \leq n^2 Q_{x,y}\left[\Delta_G>\log n\right] + \log^2 n Q_{x,y}\left[\cB_{x,y}^c\right] \leq cn^{-c'}. \label{eq:zsecondtilde} \end{split} \end{equation}
The trees $\cT_x$ and $\cT_y$ are independent, so $\tilde{Z}_x\wedge\log n$ and $\tilde{Z}_y\wedge\log n$ are independent. Therefore, from \eqref{eq:zfirsttilde} and \eqref{eq:zsecondtilde} we conclude that for two arbitrary vertices $x\neq y$,
\begin{equation*} \EE_{n,p}[Z_x Z_y] = \EE_{n,p}[Z_x]\EE_{n,p}[Z_y] + o(1) \text{ as } n\to\infty. \end{equation*}
Denote $Z=\sum_{x\in G} Z_x$. It follows from the above, together with \eqref{eq:pmaxdeg}, that
\begin{align*} \EE_{n,p}\left[Z^2\right] &= \sum_{x\in G} \EE_{n,p}[Z_x^2] + \sum_{x\neq y} \left(\EE_{n,p}[Z_x]\EE_{n,p}[Z_y] + o(1)\right) \\ &= O(n\log^2 n) + O(n^3) \PP_{n,p}[\Delta_G>\log n] +  \EE_{n,p}\left[Z\right]^2 - n \EE_{n,p}[Z_x]^2 + o(n^2) \\ &= \EE_{n,p}\left[Z\right]^2 + o(n^2) \text{ as } n\to\infty. \end{align*}
Thus $\Var{Z}=o(n^2)$ and the Chebyshev inequality implies for any $\epsilon>0$
\begin{equation*} \PP_{n,p}\left[|Z-\EE_{n,p}[Z]|>\epsilon n \right] = o(1)  \text{ as } n\to\infty. \end{equation*}
This finishes the proof of the lemma.
\end{proof}

We proceed with the proof of part $(1)$ of Proposition~\ref{prop:sizeexpconc}, i.e.~the computation of $E[|\cV^{u}|]$. First observe that
\begin{equation*} E[|\cV^{u}|] = \sum_{x\in \cC_1} P[x \text{ is vacant at time }u\rho(2-\xi)\xi n] = \sum_{x\in \cC_1} P[H_x > u\rho(2-\xi)\xi n]. \end{equation*}
The task is therefore to approximate the probabilites $P[H_x > u\rho(2-\xi)\xi n]$.

Assume that the random walk $X$ is the discrete skeleton of a simple continuous-time random walk $X^c$, i.e.~the times between jumps of $X^c$ are i.i.d.~Exponential(1). Denote by $H^c_x$ the entrance time of $x$ for this continuous-time walk and by $S_k$ the time of the $k$-th jump. It is clear that $E[S_k]=k$ and $E[H^c_x]=E[H_x]$. From \cite{AB92} or \cite[Chapter 3, Proposition 23]{AFb} we know that the distribution of the entrance time of such a continuous-time walk can be approximated by an exponential distribution, namely for all $t>0$
\begin{equation} \left| P[H^c_x > t] - e^{-\frac{t}{E[H_x]}}\right| \leq \frac{1}{\lambda_{\cC_1}E[H_x]}. \label{eq:expapprox0} \end{equation}
If $k=k(n)\to\infty$ as $n\to\infty$, by the law of large numbers $P[|S_k-k|>\epsilon k]=o(1)$ as $n\to\infty$ for all $\epsilon >0$. This implies
\begin{align*} P[H_x>k]=P[H^c_x>S_k] &= P[H^c_x>S_k,~S_k\geq (1-\epsilon) k] +P[H^c_x>S_k,~S_k<(1-\epsilon) k] \\ &\leq P[H^c_x >(1-\epsilon)k]+o(1) \mbox{ as $n\to\infty$ for all $\epsilon >0$,}\end{align*}
and similarly
\begin{equation*} P[H_x>k] \geq P[H^c_x >(1+\epsilon)k] +o(1) \mbox{ as $n\to\infty$ for all $\epsilon >0$.} \end{equation*}
We obtain $P[H_x>k]=P[H^c_x>k]+o(1)$ as $n\to\infty$, and together with the bounds \eqref{prop:gap} for $\lambda_{\cC_1}$ and \eqref{eq:epihl} for $E[H_x]$ it follows from \eqref{eq:expapprox0} that $\PP_{n,p}$-a.a.s.
\begin{equation} \left| P[H_x > u\rho(2-\xi)\xi n] - e^{-\frac{u\rho(2-\xi)\xi n}{E[H_x]}}\right| = o(1). \label{eq:expapprox1} \end{equation}

Approximating the probabilities $P[H_x > u\rho(2-\xi)\xi n]$ therefore reduces to the investigation of $E[H_x]$. We will use Proposition 3.2 from \cite{CTW11}, which states that $E[H_x]$ can be approximated in terms of the Dirichlet form of the equilibrium potential $g^{\star}$ (cf.~\eqref{eq:gstar} and \eqref{eq:Dgstar}). 

\begin{proposition}{\cite[Proposition 3.2]{CTW11}} \label{prop:ctw}
\begin{equation} \cD(g^{\star},g^{\star}) \left(1-2\sup_{y\in B(x,r)^c} |f^{\star}(y)|\right) \leq \frac{1}{E[H_x]} \leq \cD(g^{\star},g^{\star}) \frac{1}{\pi(B(x,r)^c)^2},  \label{eq:ctw} \end{equation}
where $f^{\star}(y) = 1 - \frac{E_y[H_x]}{E[H_x]}$.
\end{proposition}

To use this result, we need to control the function $f^{\star}$. To this end, we give in the next lemma a bound on the probability that the random walk on $\cC_1$ started outside $B(x,r)$ hits $x$ before some time $T$. Recall the coupling $Q_x$ from Proposition~\ref{prop:localtree}, the definition \eqref{def:xgood} of the event $\{x$~good$\}$, and the definition \eqref{def:r} of the radius $r$.

\begin{lemma} \label{lem:pH2}
There is a constant $c$, such that, for $T\in\NN$ possibly depending on $n$, 
\begin{equation*} Q_x\left[ \sup_{y\in B(x,r)^c}P_y[H_x\leq T] \leq T e^{-c r} ~\bigg|~x\text{ good}\right] \to1 \text{ as } n\to\infty. \end{equation*}
\end{lemma}

\begin{proof}
For $x$ good let $\cT$ be the infinite Poisson($\rho$)-Galton-Watson tree defined by the coupling $Q_x$ to which the neighbourhood of $x$ is isomorphic. Let $P^{\cT}_w$ be the law of the simple random walk on the tree $\cT$ started at $w\in\cT$. To bound the escape probability of random walk on a Galton-Watson tree we use \cite[Proposition 11.5]{JLT12}, which states that
\begin{equation*} \sup_{w\in\partial B(\varnothing,r)} P^{\cT}_w\left[H_{\varnothing} < \infty\right] \leq e^{-cr} \qquad \PP^0_{\cT}\text{-a.a.s.~as } r\to\infty. \end{equation*}
Since $P^{\cT}_w[H_{\varnothing} < \infty] \geq P^{\cT}_w[H_{\varnothing} < H_{B(\varnothing,r)^c}]$, this implies
\begin{equation*} \sup_{w\in\partial B(\varnothing,r)} P^{\cT}_w[H_{\varnothing} <  H_{B(\varnothing,r)^c}] \leq e^{-cr} \qquad \PP_{\cT}^0\text{-a.a.s.~as } r\to\infty. \end{equation*}
As argued before, since $Q_x[x$~good$]$ is bounded away from zero, this also holds $Q_x[~\cdot~|~x$~good$]$-a.a.s. For $x$ good, $P^{\cT}_w[H_{\varnothing} <  H_{B(\varnothing,r)^c}]=P_z[H_x<H_{B(x,r)^c}]$, where $z\in\partial B(x,r)$ is the image of $w$ under the isomorphism between $B(x,r+1)\subset G$ and $B(\varnothing,r+1)\subset\cT$. It follows that
\begin{equation*} Q_x\left[ \sup_{z\in\partial B(x,r)}P_z\left[H_x<H_{B(x,r)^c}\right]\leq e^{-cr} ~\bigg|~x\text{ good}\right] \to1 \text{ as } n\to\infty. \end{equation*}
On the way from $y\in B(x,r)^c$ to $x$, the random walk on $\cC_1$ must visit some $z\in\partial B(x,r)$. From there it either reaches $x$ or leaves $B(x,r)$ again. The probability of the first event is $Q_x[~\cdot~|~x$~good$]$-a.a.s.~bounded by $e^{-cr}$, and if the second event occurs, we can repeat the previous reasoning. But in time $T$, this procedure can be repeated at most $T$ times, leading to the required bound on $P_y[H_x\leq T]$.
\end{proof}

With \lemref{lem:pH2} we can give a bound on $\sup_{y\in B(x,r)^c}\left|f^{\star}(y)\right|$ on the left hand side of \eqref{eq:ctw}.

\begin{lemma} \label{lem:bfstar}
There are constants $c$, $c'$, such that
\begin{equation}  Q_x\left[ \sup_{y\in B(x,r)^c} \left|1 - \frac{E_y[H_x]}{E[H_x]}\right|  \leq cn^{-c'} ~\bigg|~x\text{ good}\right] \to1 \text{ as } n\to\infty.  \label{eq:bfstar} \end{equation}
\end{lemma}

\begin{proof}
Note first that by the general $O(k^3)$-bound on the expected cover time $C_G$ of a graph $G$ on $k$ vertices (see e.g.~\cite{Al79}), we have
\begin{equation} \sup_{z\in \cC_1} E_z[H_x] \leq C_{\cC_1} \leq n^3. \label{eq:bsupeh} \end{equation}

Before considering the expectation of $H_x$ with the random walk started from $y\in B(x,r)^c$, we consider the expectation of $H_x$ starting from $X_T$ for some time $T$ where the walk is well mixed. Set $T=\log^4n$. With \eqref{eq:mix}, \eqref{prop:gap}, \eqref{eq:pil} and \eqref{eq:bsupeh} we get $\PP_{n,p}$-a.a.s.~for all $z\in \cC_1$
\begin{align} \big|E_z[E_{X_T}[H_x]] - E[H_x]\big| &\leq \sum_{z'\in \cC_1} \big|P_z[X_T = z'] - \pi(z')\big| E_{z'}[H_x] \notag \\
																						&\leq \sum_{z'\in \cC_1} \frac{1}{\min_{v\in \cC_1} \pi(v)} e^{-\lambda_{\cC_1}T} E_{z'}[H_x] \label{eq:ehxt} \\
																						&\leq c n^5 \log n e^{-c'\log^2n} \leq cn^{-c'}. \notag \end{align}
By the Markov property at time $T$ and using \eqref{eq:ehxt}, $\PP_{n,p}$-a.a.s.
\begin{equation} E_z[H_x] \leq T + E_z[E_{X_T}[H_x]] \leq T + E[H_x] + cn^{-c'}. \label{eq:bfstar0} \end{equation}
With \eqref{eq:epihl} it follows that $\PP_{n,p}$-a.a.s.~for all $z\in\cC_1$
\begin{equation} \frac{E_z[H_x]}{E[H_x]} - 1 \leq (T + cn^{-c'}) \frac{1}{E[H_x]} \leq cn^{-c'}. \label{eq:bfstar1} \end{equation}
Since everything holding $\PP_{n,p}$-a.a.s.~also holds $Q_x[~\cdot~|~x$~good$]$-a.a.s., \eqref{eq:bfstar1} is enough for one side of \eqref{eq:bfstar}.

For the other side take now $y\in B(x,r)^c$ and apply the Markov property at time $T$, use \eqref{eq:ehxt} on the first line and \eqref{eq:bfstar0} for the supremum on the second line to get $\PP_{n,p}$-a.a.s.
\begin{align*} E_y[H_x] &\geq E_y[\Ind{H_x > T}E_{X_T}[H_x]] = E_y[E_{X_T}[H_x]] - E_y[\Ind{H_x \leq T} E_{X_T}[H_x]] \\
												&\geq E[H_x] - cn^{-c'} - P_y[H_x\leq T] \sup_{z\in \cC_1} E_z[H_x] \\
												&\geq E[H_x] - 2 cn^{-c'} - P_y[H_x\leq T] (T + E[H_x]). \end{align*}
This holds $\PP_{n,p}$-a.s.s., so as argued before it also holds $Q_x[~\cdot~|~x$~good$]$-a.a.s. With the bound \eqref{eq:epihl} and using \lemref{lem:pH2}, where we note that $e^{-cr} = n^{-c'}$ by \eqref{def:r}, it follows that $Q_x[~\cdot~|~x$~good$]$-a.a.s.
\begin{equation*} \frac{E_y[H_x]}{E[H_x]} - 1 \geq - cn^{-1-c'}\log n - \log^4n e^{-c'' r} \left(\frac{c'''\log^5n}{n} + 1\right)  \geq -cn^{-c'}. \end{equation*}
Together with \eqref{eq:bfstar1} this proves the lemma.
\end{proof}

Applying \lemref{lem:bfstar} in \eqref{eq:ctw} and using \lemref{lem:sumdeg}, we obtain the following approximation of the probabilites $P[H_x>u\rho(2-\xi)\xi n]$.

\begin{lemma} \label{lem:exapprox2}
For any fixed $u>0$ and every $\epsilon>0$,
\begin{equation*} Q_x\left[ \left| P\left[H_x > u\rho(2-\xi)\xi n\right] - e^{-u P^{\cT}_{\varnothing}\left[\tilde{H}_{\varnothing} > H_{B(\varnothing,r)^c}\right] \degr(\varnothing)} \right|\leq \epsilon ~\bigg|~x \text{ good}\right] \to 1 \text{ as } n\to\infty.  \end{equation*}
\end{lemma}

\begin{proof}
First recall \eqref{eq:Dgstar} and use \eqref{eq:piu} to get $\PP_{n,p}$-a.a.s.
\begin{equation} \cD(g^{\star},g^{\star}) = P_x\left[\tilde{H}_x > H_{B(x,r)^c}\right]\pi(x) \leq \frac{c \log n}{n}. \label{eq:bDgstar} \end{equation}

For the left hand approximation in \eqref{eq:ctw}, \lemref{lem:bfstar} and \eqref{eq:bDgstar} imply that $Q_x[~\cdot~|~x$~good$]$-a.a.s.
\begin{equation} \frac{1}{E[H_x]} \geq \cD(g^{\star},g^{\star}) - cn^{-1-c'}. \label{eq:fracepih1} \end{equation}

For the right hand approximation in \eqref{eq:ctw}, first recall that by \eqref{eq:sizeball} $\PP_{n,p}$-a.a.s., $|B(x,r)|\leq n^{1-\epsilon}$ for some $\epsilon>0$. Together with \eqref{eq:piu} we get $\PP_{n,p}$-a.a.s.
\begin{equation*} \pi(B(x,r)^c) \geq 1- |B(x,r)| \max_{v\in \cC_1}\pi(v) \geq1-c n^{-\epsilon}\log n. \end{equation*}
Using this and \eqref{eq:bDgstar} in \eqref{eq:ctw} yields $\PP_{n,p}$-a.a.s.
\begin{equation} \begin{split} \frac{1}{E[H_x]} &\leq  \cD(g^{\star},g^{\star}) \frac{1}{(1-c n^{-\epsilon}\log n)^2} \leq \cD(g^{\star},g^{\star}) \left(1+cn^{-\epsilon} \log n \right) \\ &\leq \cD(g^{\star},g^{\star}) + c n^{-1-\epsilon} \log^2 n \label{eq:fracepih2} \leq \cD(g^{\star},g^{\star}) + cn^{-1-c'}. \end{split} \end{equation}

Combining \eqref{eq:fracepih1} and \eqref{eq:fracepih2} we obtain that $Q_x[~\cdot~|~x$~good$]$-a.a.s.
\begin{equation*} e^{-\frac{u\rho(2-\xi)\xi n}{E[H_x]}} = e^{-u\rho(2-\xi)\xi n(\cD(g^{\star},g^{\star})+o(n^{-1}))} = e^{-u\rho(2-\xi)\xi n\cD(g^{\star},g^{\star})} + o(1). \end{equation*}
Together with \eqref{eq:expapprox1} it follows that
\begin{equation} Q_x\left[ \left| P\left[H_x>u\rho(2-\xi)\xi n\right] - e^{-u\rho(2-\xi)\xi n\cD(g^{\star},g^{\star})} \right| \leq \epsilon ~\big|~x\text{ good}\right] \to1 \text{ as } n\to\infty. \label{eq:exapprox2} \end{equation}

\lemref{lem:sumdeg} implies that $\PP_{n,p}$-a.a.s.~for $x\in \cC_1$, $\pi(x) = \frac{\degr(x)}{\rho(2-\xi)\xi n} (1+o(1))$. Recalling \eqref{eq:Dgstar}, this implies that $\PP_{n,p}$-a.a.s.
\begin{equation*} u\rho(2-\xi)\xi n\cD(g^{\star},g^{\star}) = u P_x\left[\tilde{H}_x > H_{B(x,r)^c}\right] \degr(x) + o(1). \end{equation*}
Using this in \eqref{eq:exapprox2}, and noting that if $x$ is good,
\begin{equation*}  e^{-u P_x\left[\tilde{H}_x > H_{B(x,r)^c}\right] \degr(x)} = e^{-u P^{\cT}_{\varnothing}\left[\tilde{H}_{\varnothing} > H_{B(\varnothing,r)^c}\right] \degr(\varnothing)}, \end{equation*}
finishes the proof of the lemma.
\end{proof}

\begin{proof}[Proof of part $(1)$ of Proposition~\ref{prop:sizeexpconc}]
We use the same technique as in the proof of \lemref{lem:sumdeg}: We compute the expectation of $E[|\cV^{u}|]$ under $\PP_{n,p}$ and then show that $E[|\cV^{u}|]$ is concentrated. Define the random variables
\begin{align*} W_x &=\Ind{x\in\cC_1}P[H_x > u\rho(2-\xi)\xi n], \\ \tilde{W}_x &=\Ind{|\cT|=\infty} e^{-u P^{\cT}_{\varnothing}\left[\tilde{H}_{\varnothing} > H_{B(\varnothing,r)^c}\right] \degr(\varnothing)}, \end{align*}
where the tree $\cT$ is defined by the coupling $Q_x$ from Proposition~\ref{prop:localtree}, and $\varnothing$ is the root of $\cT$.

Let us first compute the expectation of $\tilde{W}_x$ as $n\to\infty$. Since $r\to\infty$ as $n\to\infty$, and the tree $\cT$ has law $\PP_{\cT}^0$,
\begin{equation} \begin{split} \lim_{n\to\infty} \EE_{Q_x}\left[\tilde{W_x}\right] &= \lim_{n\to\infty} \EE_{\cT}^0\left[e^{-u P^{\cT}_{\varnothing}\left[\tilde{H}_{\varnothing} > H_{B(\varnothing,r)^c}\right] \degr(\varnothing)} ~\big|~ |\cT|=\infty \right] \PP_{\cT}^0\left[|\cT|=\infty \right] \\ &= \EE_{\cT}\left[e^{-u \capa_{\cT}(\varnothing)}\right]\xi. \label{eq:expwtilde} \end{split} \end{equation}

For $\epsilon>0$, define on the space of the coupling $Q_x$ the event
\begin{equation} \cA_{x,\epsilon} = \{|W_x-\tilde{W}_x|\leq \epsilon \}. \label{def:axepsilon} \end{equation}
By definitions \eqref{def:bx} and \eqref{def:xgood} of the events $\cB_x$ and $\{x$~good$\}$, on $\cB_x$ either $W_x=\tilde{W}_x=0$ or $x$ is good, i.e.~$\cA^c_{x,\epsilon}\cap\cB_x = \cA^c_{x,\epsilon}\cap\{x$~good$\}$. With \lemref{lem:exapprox2} and \eqref{eq:pbx} it follows that
\begin{equation} \begin{split} Q_x[\cA_{x,\epsilon}^c] &\leq Q_x\left[\cA_{x,\epsilon}^c,~\cB_x\right] + Q_x\left[\cB^c_x\right] \\ &\leq Q_x\left[\cA_{x,\epsilon}^c~\big|~x \text{ good}\right]Q_x\left[x \text{ good}\right] + Q_x\left[\cB^c_x\right] = o(1) \text{ as } n\to\infty. \label{eq:paxepsilon} \end{split} \end{equation}
Since $W_x$ and $\tilde{W}_x$ are bounded by $1$, this implies
\begin{equation*} \left| \EE_{Q_x}[W_x] - \EE_{Q_x}[\tilde{W}_x] \right | \leq \epsilon + Q_x[\cA_{x,\epsilon}^c] \text{ for any } \epsilon>0, \end{equation*}
and thus
\begin{equation} \EE_{Q_x}[W_x] = \EE_{Q_x}[\tilde{W}_x]+o(1) \text{ as }n\to\infty. \label{eq:wfirst} \end{equation}
With \eqref{eq:expwtilde} we conclude that
\begin{equation*} \EE_{n,p}\left[E[|\cV^{u}|]\right]= \EE_{n,p}\left[\sum_{x\in G} W_x \right] = n \EE_{Q_x}[W_x] = \xi n \EE_{\cT}\left[e^{-u \capa_{\cT}(\varnothing)}\right] + o(n) \text{ as } n\to\infty. \end{equation*}

For the concentration of $E[|\cV^{u}|]$ we use again the second moment method. Consider the coupling $Q_{x,y}$ from Proposition~\ref{prop:localtree} for two fixed vertices $x\neq y$. The random variable $\tilde{W}_z$ as well as the event $\cA_{z,\epsilon}$ for $z\in\{x,y\}$ are canonically also defined on the space of $Q_{x,y}$ when replacing $\cT$ by $\cT_z$ in the definition of $\tilde{W}_z$. Let $\cA_{x,y,\epsilon}=\cA_{x,\epsilon}\cap\cA_{y,\epsilon}$, and recall the definition \eqref{def:bxy} of the set $\cB_{x,y}$, on which either $W_z=\tilde{W}_z=0$ or $z$ is good, for both $z\in\{x,y\}$. Note that the statement of \lemref{lem:exapprox2} also holds on the space of $Q_{x,y}$ when replacing $\cT$ by $\cT_z$ for both $z\in\{x,y\}$ respectively. As in \eqref{eq:paxepsilon}, with \lemref{lem:exapprox2} and \eqref{eq:pbxy} we obtain
\begin{align*} Q_{x,y}[\cA_{x,y,\epsilon}^c] &\leq Q_{x,y}\left[\cA_{x,\epsilon}^c,~\cB_{x,y}\right] + Q_{x,y}\left[\cA_{y,\epsilon}^c,~\cB_{x,y}\right] + Q_{x,y}\left[\cB^c_{x,y}\right] \\ &\leq Q_{x,y}\left[\cA^c_{x,\epsilon}~\big|~x \text{ good}\right] + Q_{x,y}\left[\cA^c_{y,\epsilon}~\big|~y \text{ good}\right] + Q_{x,y}\left[\cB^c_{x,y}\right] = o(1) \text{ as } n\to\infty. \end{align*}
Since the $W_z$ and $\tilde{W}_z$ are bounded by $1$, it follows that
\begin{equation*} \left|\EE_{Q_{x,y}}[W_x W_y] - \EE_{Q_{x,y}}[\tilde{W}_x \tilde{W}_y] \right | \leq \epsilon + Q_{x,y}[\cA_{x,y,\epsilon}^c] \text{ for any } 1>\epsilon>0, \end{equation*}
and thus
\begin{equation} \EE_{Q_{x,y}}[W_x W_y] = \EE_{Q_{x,y}}[\tilde{W}_x \tilde{W}_y]+o(1) \text{ as }n\to\infty. \label{eq:wsecond} \end{equation}
The trees $\cT_x$ and $\cT_y$ are independent, so the random variables $\tilde{W}_x$ and $\tilde{W}_y$ are independent. Therefore, \eqref{eq:wfirst} and \eqref{eq:wsecond} imply that for arbitrary vertices $x\neq y$
\begin{equation*} \EE_{n,p}[W_x W_y] = \EE_{n,p}[W_x]\EE_{n,p}[W_y] + o(1) \text{ as } n\to\infty. \end{equation*}
Recall that $E[|\cV^{u}|]=\sum_{x\in G} W_x$. By the boundedness of the $W_x$, it follows directly from the above that 
\begin{align*} \EE_{n,p}\left[E[|\cV^{u}|]^2\right] = \EE_{n,p}\left[E[|\cV^{u}|]\right]^2 + o(n^2) \text{ as } n\to\infty. \end{align*}
Thus $\Var{E[|\cV^{u}|]}=o(n^2)$ and the Chebyshev inequality implies for any $\epsilon>0$
\begin{equation*} \PP_{n,p}\left[|E[|\cV^{u}|]-\EE_{n,p}\left[E[|\cV^{u}|]\right]|>\epsilon n \right] = o(1) \text{ as } n\to\infty. \end{equation*}
This finishes the proof of the first part of Proposition~\ref{prop:sizeexpconc}.
\end{proof}


\vspace{\baselineskip}
\subsection{Concentration of the size of the vacant set}\label{sec:concent}
To prove part (2) of Proposition~\ref{prop:sizeexpconc}, we use similar techniques as in \cite{CTW11} and \cite{CT11}. We define a sequence of i.i.d.~stationary started random walk trajectories of length $n^{\delta}$ and glue them together at the endpoints to obtain a trajectory which is, by the fast mixing of the random walk, in distribution close to the random walk on $\cC_1$ but has a different dependency structure, which allows to apply the following concentration result by \cite{McD98}.

\begin{theorem}{\cite[Theorem 3.7]{McD98}} \label{thm:conc}
Let $W=(W_1,...,W_M)$ be a familiy of random variables $W_k$ taking values in a set $\cA_k$, and let $f$ be a bounded real-valued function on $\prod \cA_k$. Let $\mu$ denote the mean of $f(W)$. Define
\begin{align*} r_k(&y_1,...,y_{k-1}) \\ &= \sup_{y,y'\in\cA_k} &\bigg| E\left[f(W)~\big|~W_k=y,~W_i=y_i ~\forall i<k\right] - E\left[f(W)~\big|~W_k=y',~W_i=y_i ~\forall i<k\right] \bigg|, \end{align*}
and let
\begin{equation*} R^2 = \sup_{y_1,...,y_{M-1}} \sum_{k=1}^M r_k^2 (y_1,...,y_{k-1}) . \end{equation*}
Then for any $t\geq0$,
\begin{equation*} P\left[|f(W)-\mu|\geq t\right] \leq 2 e^{-\frac{t^2}{R^2}}. \end{equation*}
\end{theorem}

Let us define precisely the above mentioned approximation of the random walk. Denote by $P_x^L$ the restriction of $P_x$ to $\cC_1^{L+1}$, i.e.~the law of the trajectory $(X_0,...,X_L)$ and by $P_{x,z}^L$ the law of the random walk bridge, that is $P^L_x$ conditioned on $X_L=z$. Fix $\delta>0$ and let $L=n^{\delta}$. For a given typical random graph $G$ define on an auxiliary probability space $(\hat{\Omega}, \hat{\cA}, \hat{P})$  the i.i.d.~random variables $(Z^i)_{i\geq0}$ as vertices of $\cC_1$ chosen according to the stationary measure $\pi$. Given the collection $(Z^i)$, let $(Y^i)_{i\geq1}$ be conditionally independent elements of $\cC_1^{L+1}$ such that each $(Y^i_k)_{k=0,...,L}$ is distributed according to the random walk bridge $P^L_{Z^{i-1},Z^i}$. We define the concatenation of the $Y^i$ as
\begin{equation*} \cX_t = Y^i_{t-(i-1)L}, \qquad \text{when } (i-1)L\leq t< iL. \end{equation*}
Denote by $\cP^{u}$ the law of $\cX$ on $\cC_1^{u\rho(2-\xi)\xi n+1}$ and write $P^u$ for $P^{u\rho(2-\xi)\xi n}$, that is $P$ restricted to $\cC_1^{u\rho(2-\xi)\xi n+1}$. The next lemma shows that $\cP^{u}$ approximates $P^{u}$ well if $L$ is large enough.

\begin{lemma} \label{lem:concat}
$\PP_{n,p}$-a.a.s.~the measures $\cP^{u}$ and $P^{u}$ are equivalent, and for $n$ large enough and constants $c$, $c'$ depending on $\delta$,
\begin{equation*} \left| \frac{dP^{u}}{d\cP^{u}} -1 \right| \leq c e^{-c'n^{\frac{\delta}{2}}}. \end{equation*}
\end{lemma}

\begin{proof}
Let $u'$ be the smallest number greater or equal to $u$ such that $u'\rho(2-\xi)\xi n$ is an integer multiple of $L$ and set $m=\frac{u'\rho(2-\xi)\xi n}{L}$. Since $\cP^{u}$ and $P^{u}$ are the restrictions of $\cP^{u'}$ and $P^{u'}$ to $\cC_1^{u\rho(2-\xi)\xi n+1}$, it is sufficient to prove the lemma for $\cP^{u'}$ and $P^{u'}$. Let $A$ be any measurable subset of $\cC_1^{u\rho(2-\xi)\xi n+1}$. Then by the Markov property
\begin{align} P^{u'}[A] &= \sum_{x_0,...,x_m \in \cC_1} P^{u'}\left[A~\big|~X_{iL}=x_{i},~0\leq i\leq m\right] P^{u'}\left[X_{iL}=x_{i},~0\leq i\leq m\right] \notag \\
													&= \sum_{x_0,...,x_m \in \cC_1} P^{u'}\left[A~\big|~X_{iL}=x_{i},~0\leq i\leq m\right] \pi(x_0) \prod_{k=0}^m P^L_{x_k} [X_L=x_{k+1}]. \label{eq:concat1} \end{align}
													
Next, note that $\cP^{u'}\left[X_{iL} = x_{i},~0\leq i\leq m\right]=0$ if and only if $P^{u'}\left[X_{iL}=x_{i},~0\leq i\leq m\right]=0$: One can always choose the $m$ $x_i$'s, but there might not be any way to connect them by random walk bridges, whence the probability is zero. In this case, there is also no random walk trajectory going through this points. On the other hand, when there is no such trajectory, there are also no bridges.
												
From this and the construction of the measure $\cP$ it follows that, whenever this is well-defined,
\begin{equation} \begin{split} \cP^{u'}\left[A~\big|~X_{iL} = x_{i},~0\leq i\leq m\right] &= P^{u'}\left[A~\big|~X_{iL}=x_{i},~0\leq i\leq m\right], \label{eq:concat2} \\
								\cP^{u'}\left[X_{iL} = x_{i},~0\leq i\leq m\right] &= \prod_{k=0}^m \pi(x_k).  \end{split} \end{equation}

Comparing \eqref{eq:concat1} and \eqref{eq:concat2}, it remains to control the ratio $\frac{P^L_{x} [X_L=y]}{\pi(y)}$. We use \eqref{eq:mix}, \eqref{eq:pil} and \eqref{prop:gap} to get $\PP_{n,p}$-a.a.s.
\begin{equation*} \left| \frac{P^L_{x} [X_L=y]}{\pi(y)} -1 \right| \leq \frac{1}{(\min_{z\in \cC_1} \pi(z))^2} e^{-\lambda_{\cC_1}L}\leq cn^2\log^2n e^{-\frac{c' }{\log^2n} L}. \end{equation*}
With $\frac{n^{\delta}}{\log^2n}\geq cn^{\frac{\delta}{2}}$ for $n$ large enough it follows that $\PP_{n,p}$-a.a.s.
\begin{equation*} \left( 1- cn^2\log^2n e^{-c'n^{\frac{\delta}{2}}} \right)^m \leq \frac{P^{u'}[A]}{\cP^{u'}[A]} \leq \left( 1 + cn^2\log^2n e^{-c'n^{\frac{\delta}{2}}}  \right)^m, \end{equation*}
and hence $\PP_{n,p}$-a.a.s.~$\cP^{u'}$ and $P^{u'}$ are equivalent, and the lemma follows by changing constants to accomodate the terms polynomial in $n$ and $\log n$.
\end{proof}

\begin{proof}[Proof of part $(2)$ of Proposition~\ref{prop:sizeexpconc}]
We show that for any $\delta>0$,
\begin{equation} P\left[\big||\cV^{u}|-E[|\cV^{u}|]\big| \geq n^{\frac{1}{2}+\delta}\right] \leq ce^{-c'n^{\frac{\delta}{2}}} \qquad \PP_{n,p} \text{-a.a.s.}, \label{eq:concentration} \end{equation}
which implies the statement of the proposition.

Set $m=\floor{\frac{u\rho(2-\xi)\xi n}{L}}$ and $u'=\frac{mL}{\rho(2-\xi)\xi n}$. Then $u\rho(2-\xi)\xi n-u'\rho(2-\xi)\xi n\leq L$, and so $\left||\cV^{u}|-|\cV^{u'}|\right|\leq L$. It follows that for $n$ large enough
\begin{equation} \begin{split} P\left[\big||\cV^{u}|-E[|\cV^{u}|]\big| \geq n^{\frac{1}{2}+\delta}\right] &\leq P\left[\left||\cV^{u'}|-E[|\cV^{u'}|]\right| \geq n^{\frac{1}{2}+\delta}-2L\right] \\ &\leq P\left[\left||\cV^{u'}|-E[|\cV^{u'}|]\right| \geq \frac{1}{2} n^{\frac{1}{2}+\delta}\right]. \label{eq:conc1} \end{split} \end{equation}
Let $\cU^{u'}=\cC_1\setminus\cX_{[0,mL]}$ be the vacant set left by the concatenation $\cX$, and denote by $\cE$ the expectation corresponding to $\cP$. \lemref{lem:concat} implies that $\PP_{n,p}$-a.a.s.
\begin{align*} \left|P\left[\cV^{u'} \in ~\cdot~ \right] - \cP\left[\cU^{u'} \in ~\cdot~ \right] \right| &\leq ce^{-c'n^{\frac{\delta}{2}}}, \\
								\left| E[|\cV^{u'}|] - \cE[|\cU^{u'}|] \right| \leq cn e^{-c'n^{\frac{\delta}{2}}} &\leq \frac{1}{4} n^{\frac{1}{2}+\delta}. \end{align*}
From this we obtain that $\PP_{n,p}$-a.a.s.
\begin{equation} P\left[ \left||\cV^{u'}|-E[|\cV^{u'}|]\right| \geq \frac{1}{2} n^{\frac{1}{2}+\delta}\right] \leq \cP\left[\left||\cU^{u'}|-\cE[|\cU^{u'}|]\right| \geq \frac{1}{4} n^{\frac{1}{2}+\delta}\right] + c e^{-c'n^{\frac{\delta}{2}}}. \label{eq:conc2} \end{equation}
We now apply \thmref{thm:conc} with $M=m$, $\cA_k=\cC_1^{L+1}$, $W_k=Y^k$ and $f(W)=|\cU^{u'}|$. We claim that
\begin{align*} &r_k(y_1,...,y_{k-1}) \\ &= \sup_{y,y'\in\cA_k} \left| E\left[|\cU^{u'}|~\big|~Y^k=y,~Y^i=y_i ~\forall i<k\right] - E\left[|\cU^{u'}|~\big|~Y^k=y',~Y^i=y_i ~\forall i<k\right] \right| \\ &\leq 2L. \end{align*}
Indeed, when conditioning additionally on $Y^{k+2},...,Y^m$, the only two different segments $Y^k$ and $Y^{k+1}$ can change the size of the vacant set by at most the length of two segments, and the claim follows by integrating over all possible $Y^{k+2},...,Y^m$.

Then $R^2 \leq m(2L)^2 \leq \frac{u\rho(2-\xi)\xi n}{L} 4L^2 = c n^{1+\delta}$, and \thmref{thm:conc} implies
\begin{equation*} \cP\left[\left||\cU^{u'}|-\cE[|\cU^{u'}|]\right| \geq \frac{1}{4} n^{\frac{1}{2}+\delta}\right] \leq 2 e^{-2\frac{\frac{1}{16}n^{1+2\delta}}{c n^{1+\delta}}} = ce^{-c'n^{\delta}}. \end{equation*}
This together with \eqref{eq:conc1} and \eqref{eq:conc2} proves \eqref{eq:concentration} and hence part $(2)$ of Proposition~\ref{prop:sizeexpconc}.
\end{proof}


\vspace{\baselineskip}
\section{Coupling of processes} \label{sec:coupling}

In this section we introduce a process $\bar{X}$ which satisfies the spatial Markov property described in the introduction. We derive a phase transition in the vacant set of this process, and we compare it with the simple random walk $X$ on the giant component.

Consider the following algorithm defined on an auxiliary probability space $(\tilde{\Omega}, \tilde{\cA}, \tilde{P})$ which builds an element of $\Omega_n=\cG(n)\times\{1,2,...,n\}^{\NN_0}$, that is a graph on $n$ vertices and a random walk-like process on this graph. All the random choices made in the algorithm are independent variables defined on $\tilde{\Omega}$.

\begin{algorithm} \label{algo}
At the beginning all $n$ vertices are unvisited, and all $\binom{n}{2}$ possible edges are unexplored. When the algorithm (or the so defined process) passes an unvisited vertex, this vertex is marked visited. Edges adjacent to the vertex will be explored and become either open or closed. After being explored, the state of an egde does not change.
\begin{enumerate}[(1)]
\item Start at time 0 with a uniformly chosen vertex $v_0$ among all $n$ vertices, mark it visited.
\item Being at time $k\geq0$ with current vertex $v_k$, check first if there are any unvisited vertices left:
	\begin{itemize}
	\item If there are, let any unexplored egde adjacent to $v_k$ be explored and marked open with probability $p=\frac{\rho}{n}$ and closed otherwise. All vertices $w$ such that the egde $\{v_k,w\}$ is open are called neighbours of $v_k$.
	\item If there are no unvisited vertices left, let $\{v_l\}_{l>k}$ be uniformly at random chosen vertices and terminate the algorithm (this choice of continuation of the process $v_k$ is totally arbitrary and does not influence the reasoning below).
	\end{itemize}
\item If $v_k$ has at least one neighbour, and if there are any unvisited vertices adjacent to explored edges, choose vertex $v_{k+1}$ uniformly among all neighbours of $v_k$ and mark $v_{k+1}$ visited, go to step (2) and proceed with current vertex $v_{k+1}$.
\item If $v_k$ has no neighbours or if there are no unvisited vertices adjacent to explored edges, the current component is entirely covered. Then choose vertex $v_{k+1}$ uniformly among all $n$ vertices, mark it visited, go to step (2) and proceed with current vertex $v_{k+1}$.
\end{enumerate}
\end{algorithm}

By construction, the law of the graph explored by this algorithm (edges present if they are marked open) is $\PP_{n,p}$. Let $\bar{X}$ be the process defined by $\bar{X}_k=v_k$.

It will be helpful to have two different points of view on Algorithm~\ref{algo}. The first is to look at the picture at the end of the algorithm: There is a graph $G$ and a trajectory of $\bar{X}$ covering all the vertices of the graph. Using this point of view, denote by $\bar{P}^G$ the law on $(\{1,2,...,n\}^{\NN_0},\FF_n)$ of the process $\bar{X}$ under $\tilde{P}$ conditioned on the event that the graph explored by the algorithm is $G\in\cG(n)$ (i.e.~conditioned on the random choices in Algorithm~\ref{algo} that determine the states of edges, but not on the random choices that determine the trajectory of $\bar{X}$). Under $\bar{P}^G$, the process $\bar{X}$ is, between two occurences of step (4) of the algorithm, a simple random walk on the currently explored component, started with uniform distribution on this component. Define on $\Omega_n=\cG(n)\times\{1,2,...,n\}^{\NN_0}$ the annealed measure (cf.~\eqref{def:annealed}) by
\begin{equation*} \bar{\bP}_n(A\times B)=\sum_{G\in A} \PP_{n,p}(G) \bar{P}^G(B)  \qquad \text{for } A\in\GG_n,~B\in\FF_n. \end{equation*}

The second point of view is to look at Algorithm~\ref{algo} as building the graph $G$ on-the-go. Having this in mind, the next lemma, which is crucial for the proof of \thmref{thm1}, is straightforward (cf.~\cite[Lemma 6]{CF11} for a similar statement). Let $\bar{\cV}(t) = G \setminus \bar{X}_{[0,t]}$ be the vacant set left by the process $\bar{X}$ at time $t$, defined on $(\Omega_n, \bar{\bP}_n)$. Once again we use the same notation $\bar{\cV}(t)$ for the set of vertices as well as the induced subgraph of $G$.

\begin{lemma} \label{lem:defdec}
Under $\bar{\bP}_n$ conditioned on $|\bar{\cV}(t)|=N$ the graph $\bar{\cV}(t)$ has marginal law $\PP_{N,p}$.
\end{lemma}
\begin{proof}
By construction of Algorithm~\ref{algo}, the vacant graph $\bar{\cV}(t)$ consists of the $|\bar{\cV}(t)|$ unvisited vertices at time $t$. Edges possibly connecting $\bar{\cV}(t)$ and the already visited vertices as well as all edges possibly connecting two already visited vertices are explored. So the edges eligible to be edges of $\bar{\cV}(t)$ are exactly all unexplored edges at time $t$. Because their state has not yet been decided by the algorithm, all these egdes are open with probability $\frac{\rho}{n}$, independently of what happened before, independently of each other. Therefore, the vacant graph $\bar{\cV}(t)$ is a standard Erd\H{o}s-R\'enyi random graph on $N=|\bar{\cV}(t)|$ vertices, every edge present with probability $p=\frac{\rho}{n}$, and hence it has law $\PP_{N,p}$.
\end{proof}

From \lemref{lem:defdec} and the classical results on random graphs it follows directly that the component structure of the vacant graph $\bar{\cV}(t)$ exhibits a phase transition at the time $t$ for which $|\bar{\cV}(t)|\frac{\rho}{n}=1$. To translate this phase transition to the simple random walk $X$ on the giant component $\cC_1(G)$, we need to couple $X$ to the process $\bar{X}$. We do this by first giving a coupling of $X$ and $\bar{X}$ under $P^{\cC_1}$ and $\bar{P}^G$ respectively on a fixed typical graph $G$. In Section~\ref{sec:proof} we will extend this coupling to an ``annealed'' coupling of $X$ and $\bar{X}$ under $\bP_n$ and $\bar{\bP}_n$ respectively.

\begin{proposition} \label{prop:coupling}
For $n$ large enough, for every fixed typical graph $G\in\cG(n)$ there exists a coupling $Q^G$ of $\bar{X}$ under $\bar{P}^G$ and $X$ under $P^{\cC_1(G)}$ such that
\begin{equation*}  Q^G \left[ \{X_k=\bar{X}_{k+2\log^5 n} \text{ for all }k=0,1,...,u\rho(2-\xi)\xi n \}^c\right] \leq \frac{c}{n^{c'}}. \end{equation*}
\end{proposition}

\begin{proof}
We first show that $\bar{X}$ typically is on the largest component $\cC_1$ at time $\log^{5}n$, that it mixes quickly, and then stays on $\cC_1$ until time $u\rho(2-\xi)\xi n + 2\log^{5}n$. This will allow us to identify $X$ with $\bar{X}$ in this time interval on an event of high probability.

Let $G$ be the typical graph (i.e.~a graph satisfying \eqref{prop:giant}, \eqref{prop:gap} and \eqref{prop:maxdeg}) explored by Algorithm~\ref{algo} and $\cC_1$ its giant component, i.e.~we look at the picture after completion of the algorithm. Define the probability distribution $\bar{\pi}$ on $G$ as the distribution of $\bar{X}_{2\log^{5}n}$, and view the stationary distribution $\pi$ of the random walk on $\cC_1$ as a distribution on the whole graph $G$ by setting $\pi\equiv0$ on $G \setminus \cC_1$. Denote by $||\cdot||_{\text{TV}}$ the total variation norm. Define $\tau=\min\{t\geq\log^{5}n:$ step $(4)$ of Algorithm~\ref{algo} is performed$\}$. $\tau$ is the first time after $\log^{5}n$ where $\bar{X}$ does not behave like a random walk. We show that for $n$ large enough the following properties hold for a typical graph $G$:
\begin{align} &\bar{P}^G[\bar{X}_{\log^{5}n} \notin \cC_1] \leq \frac{c}{n^{1+c'}}, \label{eq:xbarinc1} \\ 
							&\bar{P}^G[\tau \leq u\rho(2-\xi)\xi n + 2\log^{5}n] \leq \frac{c}{n^{1+c'}}, \label{eq:taujump} \\
							&\left\|\bar{\pi}-\pi\right\|_{\text{TV}} \leq \frac{c}{n^{c'}}. \label{eq:tvto0} \end{align} 

Since $G$ is typical, there is a giant component of size $\left||\cC_1| - \xi n\right| \leq n^{3/4}$, and all other components are simple (i.e.~they have at most as many edges as vertices) and of size smaller than $C \log n$. For \eqref{eq:xbarinc1}, since for $n$ large enough the random walk cannot cover $\cC_1$ in $\log^{4}n$ steps,
\begin{align} \bar{P}^G[\bar{X}_{\log^{4}n} \notin \cC_1] \leq \bar{P}^G[&\bar{X} \text{ starts on a small component and stays on small} \label{eq:psmallc}\\ &\text{components for time longer than } \log^{4} n]. \notag \end{align}
Let $N_s$ be the number of small components that $\bar{X}$ visits before reaching the giant component. By construction and since by \eqref{prop:giant} $|\cC_1|\geq (\xi -\epsilon)n$ for some $\epsilon>0$, $N_s$ is stochastically dominated by a Geometric($\xi-\epsilon$) random variable, in particular it has a finite mean. Then by the Markov inequality
\begin{equation} \bar{P}^G[N_s\geq \log n] \leq \frac{c}{\log n}. \label{eq:psmallc1} \end{equation}
Let $C_s^{(i)}$ be the cover time of the $i$-th small component covered by $\bar{X}$. The expected cover time of a graph on $k$ vertices and $m$ edges is bounded by $2m(k-1)$ (see e.g.~\cite{Al79}), so the expected cover time $\bar{E}^G[C_s^{(i)}]$ of a simple component of size smaller than $C\log n$ is bounded by $C'\log^2n$. The Markov inequality implies 
\begin{equation} \bar{P}^G\left[\sum_{i=1}^{\log n} C_s^{(i)} \geq \log^{4} n\right] \leq \frac{\log n \bar{E}^G[C_s^{(i)}]}{\log^{4} n} \leq \frac{c}{\log n}. \label{eq:psmallc2} \end{equation}
From \eqref{eq:psmallc1} and \eqref{eq:psmallc2} it follows that the probability on the right hand side of \eqref{eq:psmallc} is  smaller than $\frac{c}{\log n}$. Given $\bar{X}$ has not found $\cC_1$ after $\log^4 n$ steps, some small components are partly or entirely covered, but one can use the same line of arguments as above for the next $\log^4n$ steps to get
\begin{equation*} \bar{P}^G \left[\bar{X}_{2\log^4 n} \notin \cC_1 \mid \bar{X}_{\log^4 n} \notin \cC_1 \right] \leq \bar{P}^G \left[ \bar{X}_{\log^4 n} \notin \cC_1 \right]. \end{equation*}
Using this, we have
\begin{equation*} \bar{P}^G \left[\bar{X}_{2\log^4 n} \notin \cC_1 \right] = \bar{P}^G \left[\bar{X}_{2\log^4 n} \notin \cC_1 \mid \bar{X}_{\log^4 n} \notin \cC_1 \right] \bar{P}^G \left[ \bar{X}_{\log^4 n} \notin \cC_1 \right] \leq \bar{P}^G \left[ \bar{X}_{\log^4 n} \notin \cC_1 \right]^2. \end{equation*}
Since $\bar{X}$ cannot cover $\cC_1$ in $\log^5n$ steps we can iterate the above $\log n$ times, then
\begin{equation*} \bar{P}^G \left[ \bar{X}_{\log^5 n} \notin \cC_1 \right] \leq \left( \frac{c}{\log n} \right)^{\log n} \leq \frac{c}{n^{1+c'}}, \end{equation*}
which proves \eqref{eq:xbarinc1}.

To prove \eqref{eq:taujump} first note that $P^{\cC_1}[~\cdot~]= \sum_{z\in\cC_1}\pi(z)P^{\cC_1}_z[~\cdot~]$. With \eqref{eq:piu} it follows that
\begin{equation} \sup_{z\in\cC_1}P^{\cC_1}_z[~\cdot~] \leq \frac{1}{\min_{z\in \cC_1}\pi(z)} P^{\cC_1}[~\cdot~] \leq cn\log n P^{\cC_1}[~\cdot~]. \label{eq:supzc1} \end{equation}
Using \eqref{eq:xbarinc1} we have
\begin{align*} \bar{P}^G[\tau \leq &u\rho(2-\xi)\xi n + 2\log^5 n] \\ & \leq \sup_{z\in\cC_1}P^{\cC_1}_z[\text{cover time of $\cC_1$ is smaller than $u\rho(2-\xi)\xi n+2\log^{5}n$}] + \bar{P}^G[\bar{X}_{\log^{5}n} \notin \cC_1] \\ &\leq \sup_{z\in\cC_1}P^{\cC_1}[\text{vacant set }\cV(u\rho(2-\xi)\xi n+2\log^{5}n)\text{ is empty}] + \frac{c}{n^{1+c'}}. \end{align*}
Since adding a trajectory of length $2\log^{5}n$ can decrease the size of the vacant set by at most $2\log^{5}n = o(n)$, it follows that asymptotically $|\cV(u\rho(2-\xi)\xi n+2\log^{5}n)|=|\cV(u\rho(2-\xi)\xi n)|+o(n)$. Using \eqref{eq:supzc1}, from \eqref{eq:concentration} and part (1) of Proposition~\ref{prop:sizeexpconc} it follows that for a typical graph and $\epsilon$ small enough
\begin{equation*} \sup_{z\in\cC_1}P^{\cC_1}_z[|\cV(u\rho(2-\xi)\xi n)|< \epsilon n] \leq cn\log n P^{\cC_1}[|\cV(u\rho(2-\xi)\xi n)|< \epsilon n] \leq c'n\log n e^{-c''n^{\frac{\delta}{2}}}, \end{equation*}
where $\delta>0$ is the parameter defining the length of the random walk bridges in Section~\ref{sec:concent}. For any choice of $\delta$ we can find constants such that the above expression is smaller than $\frac{c}{n^{1+c'}}$, and \eqref{eq:taujump} follows.

For the proof of \eqref{eq:tvto0} let $P^{\cC_1}_{\mu}$ denote the law of the random walk on $\cC_1$ started at initial distribution $\mu$. When $\bar{X}$ is on $\cC_1$ at time $\log^5n$, it has then some distribution $\mu$ and it cannot cover $\cC_1$ in time $\log^5 n$. Using \eqref{eq:mix}, we thus get for every $y\in \cC_1$
\begin{align*} \left|\bar{P}^G[\bar{X}_{2\log^5n}=y] - \pi(y)\right| &\leq \bar{P}^G \left[ \bar{X}_{\log^5 n} \notin \cC_1 \right] + \sup_{\mu}\left|P^{\cC_1}_{\mu}[X_{\log^5n}=y] - \pi(y)\right| \\ &\leq \bar{P}^G \left[ \bar{X}_{\log^5 n} \notin \cC_1 \right] + \frac{1}{\min_{v\in \cC_1}\pi(v)} e^{-\lambda_{\cC_1}\log^5n}, \end{align*}
With \eqref{eq:xbarinc1}, \eqref{prop:gap} and \eqref{eq:pil}, it follows for every $y\in \cC_1$
\begin{equation}  \left|\bar{P}^G[\bar{X}_{2\log^5n}=y] - \pi(y)\right| \leq  \frac{c}{n^{1+c'}}. \label{eq:mixed} \end{equation}
We set $\pi\equiv0$ on $G \setminus \cC_1$, and by \eqref{eq:xbarinc1} and \eqref{eq:taujump}, $\bar{P}^G[\bar{X}_{2\log^5n}=y] \leq  \frac{c}{n^{1+c'}}$ for $y\in G \setminus \cC_1$. Thus \eqref{eq:mixed} holds for all $y\in G$. \eqref{eq:tvto0} follows from \eqref{eq:mixed} since by e.g.~\cite[Proposition 4.2]{LPW09} we have $\left\|\bar{\pi}-\pi\right\|_{\text{TV}} \leq n \max_{y\in G} \left|\bar{P}^G[\bar{X}_{2\log^5n}=y]-\pi(y)\right|$.

We can now define the coupling of $X$ under $P^{\cC_1}$ and $\bar{X}$ under $\bar{P}^G$. Consider again the (possibly enlarged) auxiliary probability space $(\tilde{\Omega}, \tilde{\cA}, \tilde{P})$, on which originally $\bar{X}$ was defined. On this auxiliary space we define a random variable $Y$ on $G$ with distribution $\pi$. $Y$ depends on the graph $G$ (i.e.~it depends on the random choices in Algorithm~\ref{algo} that determine the states of edges), and it may depend on the random choices that determine the trajectory of $\bar{X}$ up to time $2\log^5n$, but it is independent of all the random choices that determine the trajectory of $\bar{X}$ at times $2\log^5n+k$, $k\geq1$. By e.g.~\cite[Proposition 4.7]{LPW09} we can choose $Y$ such that $\tilde{P}[\bar{X}_{2\log^{5}n}\neq Y]=\left\|\bar{\pi}-\pi\right\|_{\text{TV}}$. By \eqref{eq:tvto0} it follows that
\begin{equation} \tilde{P}[\bar{X}_{2\log^{5}n}\neq Y]\leq \frac{c}{n^{c'}}. \label{eq:coupy} \end{equation}
Moreover, we define on $\tilde{\Omega}$ a collection $\tilde{X}^z$, $z\in G$, of independent simple random walks on $G$ started at $z$, independent of $\bar{X}$ and $Y$ (i.e.~depending only on the random choices in Algorithm~\ref{algo} that determine the  states of edges, but independent of the random choices that determine the trajectory of $\bar{X}$).

Define the process $X$ using $\bar{X}$, $Y$ and $\tilde{X}^z$ as follows,
\begin{equation} \begin{split} \begin{aligned} &\left. \begin{aligned} & X_k=\bar{X}_{k+2\log^{5}n} \text{ for } 0\leq k \leq \tau, \\ &	X_k= \tilde{X}^{\bar{X}_{\tau}}_k \text{ for } k>\tau, 	\end{aligned}\right\} &\text{if } \bar{X}_{\log^{5}n} &\in \cC_1 \text{ and } Y=\bar{X}_{2\log^{5}n}, \\
							&~X_k=\tilde{X}_k^{Y} \text{ for }k\geq0,																								&\text{if } \bar{X}_{\log^{5}n} &\in \cC_1 \text{ and } Y\neq\bar{X}_{2\log^{5}n}, \label{eq:coupling}\\
							&~X_k=\tilde{X}_k^{Y} \text{ for }k\geq0,																								&\text{if } \bar{X}_{\log^{5}n} &\notin \cC_1. \end{aligned} \end{split} \end{equation}
Let $Q^G$ denote the joint law of $\bar{X}$ and $X$ on $\{1,2,...,n\}^{2\mathbb{N}_0}$. Since $Y$ has distribution $\pi$ and $\bar{X}$ behaves like a random walk between occurences of step $(4)$ of Algorithm~\ref{algo}, in any case $X$ is a simple random walk on $\cC_1$ started stationary, so $Q^G$ is indeed a coupling of simple random walk on the giant component and the process $\bar{X}$ from Algorithm~\ref{algo} with marginal laws $P^{\cC_1}$ and $\bar{P}^G$ respectively.

By \eqref{eq:coupy}, \eqref{eq:xbarinc1} and \eqref{eq:taujump} the first case of the coupling \eqref{eq:coupling} happens with probability $\geq 1- \frac{c}{n^{c'}}$, and by \eqref{eq:taujump} also $\tau > u\rho(2-\xi)\xi n + 2\log^{5}n$ with high probability, and the statement of the proposition follows.
\end{proof}

The coupling \eqref{eq:coupling} defined in the proof of Proposition~\ref{prop:coupling} will allow us to deduce the phase transition in the vacant set left by $X$ from the phase transition in the vacant set left by $\bar{X}$. To apply the results of Section~\ref{sec:size}, we have to find the relation between the sizes of these vacant sets. This relation is given by the next lemma. Denote $\bar{\cV}^{u}=\bar{\cV}(u\rho(2-\xi)\xi n + 2 \log^5n)=G\setminus \bar{X}_{[0,u\rho(2-\xi)\xi n + 2 \log^5n]}$ and as before $\cV^{u}=\cV(u\rho(2-\xi)\xi n)=\cC_1\setminus X_{[0,u\rho(2-\xi)\xi n]}$. 

\begin{lemma} \label{lem:sizescomp}
For a sequence of typical graphs $G$ and any fixed $u>0$, with respect to the corresponding sequence of couplings $Q^G$, the random variables $|\bar{\cV}^{u}|$ and $|\cV^{u}|$ satisfy
\begin{equation*} |\bar{\cV}^{u}| = |\cV^{u}| + (1-\xi)n+o(n) \quad Q^{G}\text{-a.a.s.} \end{equation*}
\end{lemma}

\begin{proof}
Denote $\bar{\cW}^{u}=G\setminus \bar{X}_{[2\log^{5}n,u\rho(2-\xi)\xi n+2\log^{5}n]}$. Then $\left||\bar{\cV}^{u}|-|\bar{\cW}^{u}|\right|\leq 2\log^{5}n$, and for any $\epsilon>0$, $\epsilon n-2\log^{5}n\geq \frac{\epsilon}{2}n$ for $n$ large enough, thus
\begin{equation*} Q^G\left[\left| |\bar{\cV}^{u}| - |\cV^{u}| - (1-\xi)n \right|>\epsilon n\right] \leq  Q^G\left[\left| |\bar{\cW}^{u}| - |\cV^{u}| - (1-\xi)n \right|>\frac{\epsilon}{2} n\right]. \end{equation*}
By Proposition~\ref{prop:coupling}, $Q^G$-a.a.s.~the sets $\bar{\cW}^{u}$ and $\cV^{u}$ differ only by the small components of the graph $G$, i.e.~$\bar{\cW}^{u} = \cV^{u} \cup \bigcup_{i\geq2} \cC_i(G)$. By \eqref{prop:giant}, in a typical graph $G$ the total size of small components satisfies $\left|\bigcup_{i\geq2} \cC_i(G) - (1-\xi)n \right| \leq \frac{\epsilon}{2}n$ for $n$ large enough. Therefore, for every $\epsilon>0$,
\begin{equation*} Q^G\left[\left| |\bar{\cW}^{u}| - |\cV^{u}| - (1-\xi)n \right|>\frac{\epsilon}{2} n\right] \leq Q^G\left[\bar{\cW}^{u} \neq \cV^{u} \cup \bigcup_{i\geq2} \cC_i(G)\right] \to 0 \text{ as } n\to\infty. \end{equation*}
This proves the lemma.
\end{proof}


\vspace{\baselineskip}
\section{Proof of main result} \label{sec:proof}

We first extend the coupling $Q^G$ that was defined for typical graphs in Proposition~\ref{prop:coupling}. Let $Q^G$ for a non-typical graph $G$ be the joint law on $\{1,2,...,n\}^{2\mathbb{N}_0}$ of two independent processes $X$ and $\bar{X}$ under $P^{\cC_1(G)}$ and $\bar{P}^G$ respectively. We define the annealed coupling measure $\bQ_n$ on the space $\Omega_n'=\cG(n)\times\{1,2,...,n\}^{2\mathbb{N}_0}$ with the canonical coordinates $G$, $\bar{X}$, $X$ as
\begin{equation*} \bQ_n(A\times B)=\sum_{G\in A} \PP_{n,p}(G) Q^G(B), \end{equation*}
where $A\in\GG_n$ and $B=B_1\times B_2$ with $B_i\in\FF_n$ for $i=1,2$ (cf.~\eqref{def:annealed} for the definition of the $\sigma$-algebras $\GG_n$ and $\FF_n$). Then $\bQ_n$ is a coupling of the two processes $X$ and $\bar{X}$, where $X$ has marginal law $\bP_n$ and $\bar{X}$ has marginal law $\bar{\bP}_n$, and since every $G$ is $\PP_{n,p}$-a.a.s.~a typical graph the statements of Proposition~\ref{prop:coupling} and \lemref{lem:sizescomp} hold $\bQ_n$-a.a.s.

\begin{proof}[Proof of \thmref{thm1}]
For the proof we use the annealed coupling $\bQ_n$ of $X$ and $\bar{X}$. As a direct consequence of Proposition~\ref{prop:sizeexpconc} and \lemref{lem:sizescomp} we obtain that
\begin{equation*} |\bar{\cV}^{u}| = \xi n \EE_{\cT}\left[e^{-u \capa_{\cT}(\varnothing)}\right] + (1-\xi)n + o(n) \qquad \bQ_n \text{-a.a.s.} \end{equation*}
It follows from \lemref{lem:defdec} and the classical results on random graphs that the graph $\bar{\cV}^{u}=G\setminus\bar{X}_{[0,u\rho(2-\xi)\xi n + 2\log^{5}n]}$ exhibits a phase transition at the value $u$ such that  $\lim_{n\to\infty}|\bar{\cV}^{u}| \frac{\rho}{n}=1$. This value is the solution $u_{\star}$ of the equation
\begin{equation} \rho \xi \EE_{\cT}\left[e^{-u \capa_{\cT}(\varnothing)}\right]+\rho(1-\xi)=1. \label{eq:ustar} \end{equation}
$\bar{\cV}^{u}$ has therefore $\bQ_n$-a.a.s.~a unique giant component $\cC_1(\bar{\cV}^u)$ of size $\zeta(u,\rho)n+o(n)$ and all other components of size smaller than $\bar{C}\log n$ if $u<u_{\star}$, and it has $\bQ_n$-a.a.s.~all components of size smaller than $\bar{C}\log n$ for $u>u_{\star}$, where $\bar{C}>0$ is some fixed constant. For $u<u_{\star}$, the constant $\zeta(u,\rho)$ is given as the unique solution in $(0,1)$ of the equation
\begin{equation} \exp\left\{-\zeta \left(\rho \xi \EE_{\cT}\left[e^{-u \capa_{\cT}(\varnothing)}\right]+\rho(1-\xi)\right)\right\}=1-\zeta. \label{eq:zeta} \end{equation}
It remains to translate this phase transition to the vacant graph $\cV^u$ of the random walk on the giant component.

Let us first translate the phase transition to the subgraph induced by the slightly enlarged set $\bar{\cV}^{u} \cup \bar{X}_{[0,2\log^{5}n]}$. Adding one vertex of degree $d$ in $G$ to the graph $\bar{\cV}^{u}$ can merge at most $d$ components of $\bar{\cV}^{u}$. By \eqref{prop:maxdeg} the degree $d$ is $\bQ_n$-a.a.s.~bounded by $\log n$, so adding the vertices of $\bar{X}_{[0,2\log^{5}n]}$ can $\bQ_n$-a.a.s.~merge at most $2\log^{6}n$ components. It follows that $\bQ_n$-a.a.s., by adding $\bar{X}_{[0,2\log^{5}n]}$ to $\bar{\cV}^{u}$, any component of size smaller than $\bar{C}\log n$ in $\bar{\cV}^{u}$ can either merge with the giant component if there is one, or it can become a component of size at most $2\bar{C}\log^{7}n$. Also, in the supercritical phase the giant component can $\bQ_n$-a.a.s.~grow by at most $2\bar{C}\log^{7}n=o(n)$. Therefore, the graph induced by $\bar{\cV}^{u} \cup \bar{X}_{[0,2\log^{5}n]}$ exhibits a phase transition at $u_{\star}$ with the same size $\zeta(u,\rho)n+o(n)$ of the giant component for $u<u_{\star}$, and with the bound $2\bar{C}\log^{7}n$ for the size of the second largest component for $u<u_{\star}$ and the largest component for $u>u_{\star}$.

Recall that $\bar{\cW}^{u}$ denotes the set $G\setminus \bar{X}_{[2\log^{5}n,u\rho(2-\xi)\xi n+2\log^{5}n]}$ as well as the induced subgraph. We have the following inclusions of sets and induced subgraphs in $G$,
\begin{equation*} \bar{\cV}^{u} \subset \bar{\cW}^{u} \subset \bar{\cV}^{u} \cup \bar{X}_{[0,2\log^{5}n]}. \end{equation*}
Consider the vacant set $\cV^u\subset \cC_1$ of the random walk $X$ on the giant component. By Proposition~\ref{prop:coupling} and since every graph is $\PP_{n,p}$-a.a.s.~a typical graph, we have $\bQ_n$-a.a.s.~$\bar{\cW}^{u} = \cV^u \cup \bigcup_{i\geq2} \cC_i(G)$. It follows that
\begin{align} \bar{\cV}^{u} &\subset \cV^u \cup \bigcup_{i\geq2} \cC_i(G) \quad \bQ_n\text{-a.a.s.} \label{eq:inclvs1} \\ \cV^u &\subset \bar{\cV}^{u} \cup \bar{X}_{[0,2\log^{5}n]} \quad \bQ_n\text{-a.a.s.} \label{eq:inclvs2} \end{align}

Note that $\bQ_n$-a.a.s.~the union $\bigcup_{i\geq2} \cC_i(G)$ of all components of $G$ except the largest are exactly all small components of size smaller than $C\log n$. From this and \eqref{eq:inclvs1} it follows that $|\cC_1(\cV^u)|$ is $\bQ_n$-a.a.s.~bounded from below by $|\cC_1(\bar{\cV}^u)|$ whenever $|\cC_1(\bar{\cV}^u)|$ is larger than of order $\log n$. From \eqref{eq:inclvs2} it follows that $|\cC_1(\cV^u)|$ is $Q_n$-a.a.s.~bounded from above by $\left|\cC_1\left(\bar{\cV}^{u} \cup \bar{X}_{[0,2\log^{5}n]}\right)\right|$. The respective phase transitions in $\bar{\cV}^u$ and $\bar{\cV}^{u} \cup \bar{X}_{[0,2\log^{5}n]}$ thus immediately imply the statements \eqref{eq:thmeq1} and \eqref{eq:thmeq3} of \thmref{thm1}.

To prove \eqref{eq:thmeq2}, i.e.~the uniqueness of the giant component in the supercritical phase, fix $u<u_{\star}$ and let $\cL_n$ be the event that there are two distinct components $\cC_a$ and $\cC_b$ in $\cV^u$ both of size strictly larger than $2\bar{C}\log^{7}n$, with $\bar{C}$ as defined below \eqref{eq:ustar}. We show that $\bQ_n[\cL_n]\to0$ as $n\to\infty$, which proves \eqref{eq:thmeq2}. First note that if $\cL_n$ happens, then either $\cC_a\cap\bar{\cV}^u$ and $\cC_b\cap\bar{\cV}^u$ are distinct components in $\bar{\cV}^u$ or the inclusion in \eqref{eq:inclvs1} does not hold, which is unlikely, so
\begin{equation*} \bQ_n[\cL_n] \leq \bQ_n\left[\cL_n,~\mbox{$\cC_a\cap\bar{\cV}^u$ and $\cC_b\cap\bar{\cV}^u$ are distinct components in $\bar{\cV}^u$}\right] +o(1) \text{ as } n\to\infty. \end{equation*}
But if $\cC_a\cap\bar{\cV}^u$ and $\cC_b\cap\bar{\cV}^u$ are distinct components in $\bar{\cV}^u$, at least one of $\cC_a\cap\bar{\cV}^u$ or $\cC_b\cap\bar{\cV}^u$ is subset of $\bigcup_{i\geq2} \cC_i(\bar{\cV}^u)$, which is a union of components that are $\bQ_n$-a.a.s.~all of size smaller than  $\bar{C}\log n$. On the other hand by \eqref{eq:inclvs2}, $\cC_a \subset \left(\cC_a\cap\bar{\cV}^u\right) \cup \bar{X}_{[0,2\log^{5}n]}$, and as discussed before this last union cannot be larger than $2\bar{C}\log^{7}n$ if $\cC_a\cap\bar{\cV}^u$ consists only of components of size smaller than $\bar{C}\log n$. Thus
\begin{align*} \bQ_n[\cL_n] &\leq \bQ_n\left[\cL_n,~\mbox{$\cC_a\cap\bar{\cV}^u$ or $\cC_b\cap\bar{\cV}^u$ is subset of $\bigcup_{i\geq2} \cC_i(\bar{\cV}^u)$}\right] +o(1) \\
													&\leq \bQ_n\left[\mbox{at least one of the $\cC_i(\bar{\cV}^u),~{i\geq2}$, is larger than $\bar{C}\log n$}\right] +o(1) \\ &= o(1) \text{ as } n\to\infty. \end{align*}
This proves \eqref{eq:thmeq2}.

To see that the critical parameter $u_{\star}$ coincides with the critical parameter $u_{\star}$ of random interlacements on a Poisson($\rho$)-Galton-Watson tree conditioned on non-extinction, it suffices to notice that the characterizing equations \eqref{eq:ustar} and \eqref{eq:ustarri} of these two parameters are the same.
\end{proof}

\vspace{\baselineskip}

\renewcommand\MR[1]{\relax\ifhmode\unskip\spacefactor3000
\space\fi \MRhref{#1}{#1}}
\renewcommand{\MRhref}[2]%
{\href{http://www.ams.org/mathscinet-getitem?mr=#1}{MR #2}}
\providecommand{\href}[2]{#2}
\bibliographystyle{amsalpha}
\bibliography{references}

\newcommand{\etalchar}[1]{$^{#1}$}
\providecommand{\bysame}{\leavevmode\hbox to3em{\hrulefill}\thinspace}
\providecommand{\MR}{\relax\ifhmode\unskip\space\fi MR }
\providecommand{\MRhref}[2]{%
  \href{http://www.ams.org/mathscinet-getitem?mr=#1}{#2}
}
\providecommand{\href}[2]{#2}
\begin{thebibliography}{AKL{\etalchar{+}}79}

\bibitem[AB92]{AB92}
David~J. Aldous and Mark Brown, \emph{Inequalities for rare events in
  time-reversible {M}arkov chains. {I}}, Stochastic inequalities ({S}eattle,
  {WA}, 1991), IMS Lecture Notes Monogr. Ser., vol.~22, Inst. Math. Statist.,
  Hayward, CA, 1992, pp.~1--16. \MR{1228050}

\bibitem[AF]{AFb}
David~J. Aldous and James~A. Fill, \emph{Reversible markov chains and random
  walks on graphs}, Book in preparation,
  \url{http://www.stat.berkeley.edu/aldous/RWG/book.html}.

\bibitem[AKL{\etalchar{+}}79]{Al79}
Romas Aleliunas, Richard~M. Karp, Richard~J. Lipton, L{\'a}szl{\'o} Lov{\'a}sz,
  and Charles Rackoff, \emph{Random walks, universal traversal sequences, and
  the complexity of maze problems}, 20th {A}nnual {S}ymposium on {F}oundations
  of {C}omputer {S}cience ({S}an {J}uan, {P}uerto {R}ico, 1979), IEEE, New
  York, 1979, pp.~218--223. \MR{598110}

\bibitem[AN72]{AN72}
Krishna~B. Athreya and Peter~E. Ney, \emph{Branching processes},
  Springer-Verlag, New York, 1972, Die Grundlehren der mathematischen
  Wissenschaften, Band 196. \MR{0373040}

\bibitem[BKW06]{BKW06}
Itai {Benjamini}, Gady {Kozma}, and Nicholas {Wormald}, \emph{{The mixing time
  of the giant component of a random graph}},
  \href{http://arxiv.org/abs/math/0610459}{arXiv:math/0610459} (2006).

\bibitem[Bol01]{Bol01}
B{\'e}la Bollob{\'a}s, \emph{Random graphs}, second ed., Cambridge Studies in
  Advanced Mathematics, vol.~73, Cambridge University Press, Cambridge, 2001.
  \MR{1864966}

\bibitem[BS08]{BS08}
Itai Benjamini and Alain-Sol Sznitman, \emph{Giant component and vacant set for
  random walk on a discrete torus}, J. Eur. Math. Soc. (JEMS) \textbf{10}
  (2008), no.~1, 133--172. \MR{2349899}

\bibitem[CF11]{CF11}
Colin Cooper and Alan Frieze, \emph{Component structure of the vacant set
  induced by a random walk on a random graph}, Proceedings of the
  {T}wenty-{S}econd {A}nnual {ACM}-{SIAM} {S}ymposium on {D}iscrete
  {A}lgorithms (Philadelphia, PA), SIAM, 2011, pp.~1211--1221. \MR{2858394}

\bibitem[{\v{C}}T11]{CT11}
Ji\v{r}\'{i} {\v{C}}ern{\'y} and Augusto {Teixeira}, \emph{{Critical window for
  the vacant set left by random walk on random regular graphs}},
  \href{http://arxiv.org/abs/1101.1978}{arXiv:1101.1978} (2011).

\bibitem[{\v{C}}T12]{CT12}
Ji\v{r}\'{i} {\v{C}}ern{\'y} and Augusto Teixeira, \emph{From random walk
  trajectories to random interlacements, expanded lecture notes for the
  \uppercase{XV} brazilian school of probability in 2011}, Ensaios
  Matem\'aticos \textbf{23} (2012).

\bibitem[{\v{C}}TW11]{CTW11}
Ji\v{r}\'{i} {\v{C}}ern{\'y}, Augusto {Teixeira}, and David {Windisch},
  \emph{{Giant vacant component left by a random walk in a random d-regular
  graph}}, Annales de L'Institut Henri Poincare Section Physique Theorique
  \textbf{47} (2011), 929--968.

\bibitem[Dur10]{Dur10}
Rick Durrett, \emph{Random graph dynamics}, Cambridge Series in Statistical and
  Probabilistic Mathematics, Cambridge University Press, Cambridge, 2010.
  \MR{2656427 (2011c:05308)}

\bibitem[ER61]{ER61}
Paul Erd{\H{o}}s and Alfr{\'e}d R{\'e}nyi, \emph{On the evolution of random
  graphs}, Bull. Inst. Internat. Statist. \textbf{38} (1961), 343--347.
  \MR{0148055}

\bibitem[HM12]{HM12}
Hamed Hatami and Michael Molloy, \emph{The scaling window for a random graph
  with a given degree sequence}, Random Structures Algorithms \textbf{41}
  (2012), no.~1, 99--123. \MR{2943428}

\bibitem[Hof08]{vdH08}
Remco Van~Der Hofstad, \emph{Random graphs and complex networks}, Lecture notes
  in preparation, \url{http://www.win.tue.nl/~rhofstad/NotesRGCN.pdf}, 2008.

\bibitem[J{\L}R00]{JLR00}
Svante Janson, Tomasz {\L}uczak, and Andrzej Rucinski, \emph{Random graphs},
  Wiley-Interscience Series in Discrete Mathematics and Optimization,
  Wiley-Interscience, New York, 2000. \MR{1782847}

\bibitem[JLT12]{JLT12}
M.~{Jara}, C.~{Landim}, and A.~{Teixeira}, \emph{{Universality of trap models
  in the ergodic time scale}},
  \href{http://arxiv.org/abs/1208.5675}{arXiv:1208.5675} (2012).

\bibitem[LP12]{LP}
R.~Lyons and Y.~Peres, \emph{Probability on trees and networks}, Cambridge
  University Press. In preparation, current version available at
  \url{http://mypage.iu.edu/~rdlyons/}, 2012.

\bibitem[LPW09]{LPW09}
David~A. Levin, Yuval Peres, and Elizabeth~L. Wilmer, \emph{Markov chains and
  mixing times}, American Mathematical Society, Providence, RI, 2009, With a
  chapter by James G. Propp and David B. Wilson. \MR{2466937}

\bibitem[McD98]{McD98}
Colin McDiarmid, \emph{Concentration}, Probabilistic methods for algorithmic
  discrete mathematics, Algorithms Combin., vol.~16, Springer, Berlin, 1998,
  pp.~195--248. \MR{1678578}

\bibitem[Szn10]{Sz10}
Alain-Sol Sznitman, \emph{Vacant set of random interlacements and percolation},
  Ann. of Math. (2) \textbf{171} (2010), no.~3, 2039--2087. \MR{2680403}

\bibitem[Tas10]{Tas10}
Martin Tassy, \emph{Random interlacements on {G}alton-{W}atson trees},
  Electron. Commun. Probab. \textbf{15} (2010), 562--571. \MR{2737713}

\bibitem[Tei09]{Tei09}
Augusto Teixeira, \emph{Interlacement percolation on transient weighted
  graphs}, Electron. J. Probab. \textbf{14} (2009), no. 54, 1604--1628.
  \MR{2525105}

\bibitem[TW11]{TW11}
Augusto Teixeira and David Windisch, \emph{On the fragmentation of a torus by
  random walk}, Comm. Pure Appl. Math. \textbf{64} (2011), no.~12, 1599--1646.
  \MR{2838338}

\end{thebibliography}

\end{document}